\documentclass[11pt]{amsart}
\usepackage{amsmath}
\usepackage{amssymb}
\usepackage{amsfonts}
\usepackage{amsthm}
\usepackage{enumerate}
\usepackage[mathscr]{eucal}
\usepackage{graphicx}
\usepackage[dvipdfm, bookmarksnumbered, bookmarksopen, colorlinks, citecolor=blue, linkcolor=blue]{hyperref}

\allowdisplaybreaks

\newtheorem{theorem}{Theorem}[section]
\newtheorem{lemma}[theorem]{Lemma}
\newtheorem{corollary}[theorem]{Corollary}

\newtheorem*{conj}{Conjecture}
\theoremstyle{definition}
\newtheorem{definition}[theorem]{Definition}

\theoremstyle{remark}
\newtheorem{remark}[theorem]{Remark}

\numberwithin{equation}{section}

\def\intslash{\rlap{\kern  .32em $\mspace {.5mu}\backslash$ }\int}
\def\qsl{{\rlap{\kern  .32em $\mspace {.5mu}\backslash$ }\int_{Q_x}}}

\def\B{{\text{\rm B}}}

\def\R{{\mathbb R}}
\def\Z{{\mathbb Z}}
\def\C{{\mathbb C}}
\def\M{{\mathcal M}}
\def\N{{\mathcal N}}

\def\F{{\mathcal F}}
\def\A{{\mathcal A}}

\def\tens{\overline{\otimes}}
\def\T{\mathbb T}

\def\wh{\widehat}

\def\S{\mathcal S}

\def\Re{\text{Re}}
\def\Im{\text{Im}}

\def\pari{\partial}
\def\Ga{\Gamma}

\def\supp{{\text{\rm supp}}}

\def\inn#1#2{\langle#1,#2\rangle}

\def\rta{\rightarrow}

\def\lc{\lesssim}

\def\alp{\alpha}

\def\del{\delta}             
\def\eps{\varepsilon}

\def\tet{\theta}

\def\lam{\lambda}            \def\Lam{\Lambda}

\def\si{\sigma}              
\def\vphi{\varphi}
             
\def\om{\omega}              \def\Om{\Omega}
\def\fr{\frac}

\def\ups{\upsilon}
\newcommand{\Be}{\begin{equation}}
\newcommand{\Ee}{\end{equation}}
\newcommand{\Bes}{\begin{equation*}}
\newcommand{\Ees}{\end{equation*}}
\newcommand{\Bsp}{\begin{split}}
\newcommand{\Esp}{\end{split}}
\newcommand{\Bm}{\begin{multline}}
\newcommand{\Em}{\end{multline}}
\newcommand{\Bea}{\begin{eqnarray}}
\newcommand{\Eea}{\end{eqnarray}}
\newcommand{\Beas}{\begin{eqnarray*}}
\newcommand{\Eeas}{\end{eqnarray*}}
\newcommand{\Benu}{\begin{enumerate}}
\newcommand{\Eenu}{\end{enumerate}}
\newcommand{\Bi}{\begin{itemize}}
\newcommand{\Ei}{\end{itemize}}
\begin{document}

\title[Noncommutative Bochner-Riesz means]{Sharp estimates of noncommutative Bochner-Riesz means on two-dimensional quantum tori}

\author{Xudong Lai}
\address{Xudong Lai: Institute for Advanced Study in Mathematics, Harbin Institute of Technology, Harbin, 150001, People's Republic of China}
\email{xudonglai@hit.edu.cn}
\thanks{This work was supported by National Natural Science Foundation of China (No. 11801118, No. 12071098) and China Postdoctoral Science Foundation (No. 2017M621253, No. 2018T110279).}

\subjclass[2010]{Primary 46L52, 46L51, Secondary 46B15, 42B25}



\keywords{Noncommutative $L_p$ space, Bochner-Riesz means, Quantum tori, Kakeya maximal operator, Fourier series, Square function,  Rotational algebras}

\begin{abstract}
In this paper, we establish the full $L_p$ boundedness of noncommutative Bochner-Riesz means on two-dimensional quantum tori, which completely resolves an open problem raised in \cite{CXY13} in the sense of the $L_p$ convergence for two dimensions. The main ingredients are sharp estimates of noncommutative Kakeya maximal functions and geometric estimates in the plane.
We make the most of noncommutative theories of maximal/square functions, together with microlocal decompositions in both proofs of sharper estimates of Kakeya maximal functions and Bochner-Riesz means.
\end{abstract}

\maketitle

\section{Introduction and main results}
Inspired by operator algebras, harmonic analysis, noncommutative geometry and quantum probability, noncommutative harmonic analysis has rapidly developed recently (see e.g. \cite{Cad18,CXY13,GJP19,GJP17,HLW21,JMP14,JMP18,MSX20,
Mei07,MP09,Par09,XXX16,XXY18}).
The purpose of this paper is to study the noncommutative Bochner-Riesz means.
We start by introducing the classical Bochner-Riesz means in Euclidean spaces.

It is well-known that the boundedness and convergence of Bochner-Riesz means are among the most important problems in harmonic analysis. The study of Bochner-Riesz means can also be regarded as making precise the sense in which the Fourier inversion formula holds.
Recall that the Bochner-Riesz means on the usual torus $\T^d$ are defined by
\Be\label{e:21BRtorus}
\B_R^\lam(f)(x)=\sum_{m\in\Z^d}(1-\fr{|m|^2}{R^2})^\lam_+\wh{f}(m)e^{2\pi i\inn{m}{x}}
\Ee
where $\lam\geq0$, $R>0$, $(x)_+=\max\{x,0\}$ and $\wh{f}(m)=\int_{\T^d}f(x)e^{-2\pi i\inn{m}{x}}dx$.
The central topic of Bochner-Riesz means is to  seek the optimal range of $\lam$ such that $\B_R^\lam(f)$ converges to $f$ in some sense.
In particular, the problem of the $L_p$ convergence  turns out to  show  $(1-\fr{|\cdot|^2}{R^2})_+^\lam:\Z^d\rta\R$ is a uniform $L_p$  Fourier multiplier in $R>0$, which can be formulated as the so called Bochner-Riesz conjecture as follows (see e.g. \cite{Ste93}).
\begin{conj}
Suppose $\lam>0$ and $\fr{2d}{d+1+2\lam}<p<\fr{2d}{d-1-2\lam}$. Then we have
\Bes
\sup_{R>0}\|\B_R^\lam(f)\|_{L_p(\T^d)}\lc\|f\|_{L_p(\T^d)}.
\Ees
\end{conj}

One can also define the Bochner-Riesz means on $\R^d$ by
\Be\label{e:21BRR}
B_R^\lam(f)(x)=\int_{\R^d}\big(1-\tfrac{|\xi|^2}{R^2}\big)^\lam_+\wh{f}(\xi)e^{2\pi i\inn{\xi}{x}}d\xi,
\Ee
where $\wh{f}$ is the Fourier transform of $f$ on $\R^d$.
By the standard transference technique (see e.g. \cite{Gra14C}), the uniform $L_p$ boundedness of $\B_R^\lam$ in $R>0$ on $\T^d$  is equivalent to that of $B_R^\lam$ on $\R^d$. Because of this equivalence, in modern literature, researchers prefer to study the Bochner-Riesz means on $\R^d$.

The study of Bochner-Riesz means originated from S. Bochner \cite{Boc36}. The necessity of the conditions of $\lam$ and $p$ in the Bochner-Riesz conjecture was given by C.  Herz \cite{Her54}. In dimension two, this conjecture  has been completely resolved  by L. Carleson and P. Sj\"olin \cite{CS72}, independently later by C. Fefferman \cite{Fef73}, L. H\"ormander \cite{Hor73} and  A. C\'ordoba \cite{Cor79}.
When dimension $d\geq 3$, the Bochner-Riesz conjecture is still open. We refer to some substantial progress in \cite{BG11,GHI19,Lee04,TV00,TVV98} and the references therein.

Concerning the pointwise convergence of $B_R^\lam(f)$, it is natural to investigate the maximal Bochner-Riesz means defined by $B_*^\lam(f)(x)=\sup_{R>0}|B_R^\lam(f)(x)|$.
It is conjectured that $B_*^\lam$ is bounded on $L_p(\R^d)$ for $\lam>0$ and $\fr{2d-1}{d+2\lam}<p<\fr{2d}{d-1-2\lam}$ (see \cite{Tao02}), where the range of $p\leq 2$ is different from that of the Bochner-Riesz conjecture.
It is clear that the study of $B_*^\lam$ is hander than that of $B^\lam$. Up to now, this maximal Bochner-Riesz conjecture is even open for two dimensions. For some important progress, we refer the reader to \cite{Car83,LW20,Tao02} in the two-dimensional case and \cite{Lee04,Lee18,Ste58} for higher dimensions.

It should be pointed out that the study of Bochner-Riesz means is also quite related to several conjectures in harmonic analysis: Fourier restriction conjecture, local smoothing conjecture, maximal Kakeya function conjecture and Kakeya set conjecture (see e.g. \cite{Tao99}).
To investigate the Bochner-Riesz conjecture, except some fundamental theories of maximal operators, Calder\'on-Zygmund operators, oscillatory integral operators, etc,  researchers in harmonic analysis have invented many new and deep tools: bilinear or multilinear Fourier restriction (see \cite{BG11,Lee04,TVV98}), incident geometry (see  \cite{Bou99,Wol95}), decoupling and polynomial partitioning (see \cite{GHI19,LW20}) in the last two decades.
These new methods not only greatly improve the ranges of $\lam$ and $p$ in the study of Bochner-Riesz means, but also open new promising research directions in harmonic analysis.

On the other hand, many useful theories in harmonic analysis, such as Littlewood-Paley-Stein square functions, Hardy-Littlewood maximal operators, duality of $H_1$-BMO, Calder\'on-Zygmund operators and multiplier operators, have been successfully transferred to the noncommutative setting (see e.g. \cite{Cad18,HLX20,HWW19,Mei07,Md17,Par09,Xu07}).
Motivated by the development of this noncommutative harmonic analysis, the noncommutative Bochner-Riesz means on quantum tori have been investigated partially by Z. Chen, Q. Xu and Z. Yin \cite{CXY13} with limited indexes of $\lam$ and $p$.
Due to the lack of commutativity, the study of noncommutative Bochner-Riesz means seems to be more challenging.
For example, Z. Chen, Q. Xu and Z. Yin \cite{CXY13} established the boundedness of the maximal Bochner-Riesz means on the $L_p$ space over quantum tori for $\lam>(d-1)|\fr12-\fr1p|$, which is an analogue of a classical result by E. M. Stein \cite{Ste58}, but with a much more technical proof.
Compared with the fruitful theories of commutative Bochner-Riesz means, the noncommutative Bochner-Riesz means deserve to be investigated further.
This leads to a natural question that whether we can transfer  the modern and powerful tools mentioned before  into the noncommutative setting and apply them to studying noncommutative Bochner-Riesz means.
Since the study of two-dimensional Bochner-Riesz means is relatively simple (note the Bochner-Riesz conjecture is  resolved in this case), in this paper we focus on two dimensions and our main purpose is to  obtain the full boundedness of noncommutative Bochner-Riesz means on quantum tori by developing a new tool---the noncommutative Kakeya maximal function.

Quantum tori are also known as noncommutative tori or rotational algebras (see \cite{Rie90}).
One can regard quantum tori as analogues of usual tori.
Quantum tori are basic examples in operator algebras (see \cite{Dav96}) and  are interesting objects in noncommutative geometry which have been extensively studied (see e.g. \cite{Con94,Rie90,Var06}).
The research of analysis on quantum tori was started in \cite{Spe92,Wea96,Wea01} and the first systematic work of harmonic analysis on quantum tori was given later in \cite{CXY13}.
For recent work related to quantum tori in the direction of noncommutative analysis, we refer to see \cite{JMP14,JRZ18,MSX19,Ric16,XXX16,XXY18} and the references therein.

To illustrate our main results, we should give the definition of quantum torus.
Suppose that $d\geq 2$, $\tet=(\tet_{k,j})_{1\leq k,j\leq d}$ is a real skew symmetric $d\times d$ matrix.
The $d$-dimensional noncommutative torus $\A_\tet$ is a universal $C^*$-algebra generated by $d$ unitary operators $U_1,\cdots,U_d$ satisfying the following commutation relation:
\Bes
U_kU_j=e^{2\pi i\tet_{k,j}}U_jU_k,\quad 1\leq k,j\leq d.
\Ees
A well-known fact is that $\A_\tet$ admits a faithful tracial state $\tau$.
Define $\T_\tet^d$ the weak $*$-closure of $\A_\tet$ in the GNS representation of $\tau$.
We call $\T_\tet^d$ the $d$-dimensional quantum torus. The state $\tau$ also extends to a normal faithful state on $\T_\tet^d$, which will be denoted again by $\tau$.
Notice that when $\tet=0$, $\A_\tet=C(\T^d)$ and $\T_\tet^d=L_\infty(\T^d)$. Thus quantum torus $\T_\tet^d$ is a deformation of classical torus $\T^d$.
Let $L_p(\T_\tet^d)$ be the noncommutative space associated to pairs $(\T_\tet^d,\tau)$ with the $L_p$ norm given by $\|x\|_{L_p(\T_\tet^d)}=(\tau(|x|^p))^{1/p}$.

In the following, we consider the Bochner-Riesz means on quantum tori which are defined by
\Be\label{e:21BRqtorusi}
\B^\lam_R(f)=\sum_{m\in\Z^d}\big(1-|\tfrac{m}{R}|^2\big)^\lam_+\hat{f}(m)U^m,\quad f\in L_p(\T^d_\tet),
\Ee
where $U=(U_1,\cdots,U_d)$, $U^m=U_1^{m_1}\cdots U_d^{m_d}$ and $\wh{f}(m)=\tau((U^m)^*f)$.
A fundamental problem raised in \cite[Page 762]{CXY13} is that in which sense the Bochner-Riesz means converge back to $f$. In this paper we consider this problem in two-dimensional case and state our main results as follows.

\begin{theorem}\label{t:21BRqtorusi}
Suppose $0<\lam<\infty$ and $\fr{4}{3+2\lam}<p<\fr{4}{1-2\lam}$. Let $\B_R^\lam$ be the Bochner-Riesz means  defined in \eqref{e:21BRqtorusi} for $d=2$. Then we have
\Bes
\sup_{R>0}\|\B_R^\lam(f)\|_{L_p(\T_\tet^2)}\lc\|f\|_{L_p(\T^2_\tet)}.
\Ees
Consequently for $f\in L_p(\T_\tet^2)$, $\B_R^\lam(f)$ converges to $f$ in $L_p(\T_\tet^2)$ as $R\rta\infty$.
\end{theorem}

This theorem is in fact a noncommutative version of the two-dimensional Bochner-Riesz conjecture. Thus we completely resolve an open problem raised in \cite{CXY13} in the sense of the $L_p$ convergence for two dimensions.
In the following, we briefly introduce the strategy used in the proof of Theorem \ref{t:21BRqtorusi}.

Notice that $\B_R^\lam(f)$ is fully noncommutative and its analysis seems to be rather difficult. Nevertheless there is a clever trick that transfers the problem of multiplier operator on quantum tori to the operator-valued setting on usual tori (see \cite{CXY13}). Hence using this method, the $L_p$ boundedness of Bochner-Riesz means on quantum tori can be reduced to that of the operator-valued Bochner-Riesz means on usual tori $\T^d$ (see Theorem \ref{t:21BRtorus}).
Next we use the noncommutative transference of multiplier (see Theorem \ref{t:21transft}), we can transfer the study of the operator-valued Bochner-Riesz means on $\T^d$ to that on $\R^d$, in which case we can do analysis based on some known noncommutative theories of harmonic analysis.
However to establish the full boundedness of two-dimensional operator-valued Bochner-Riesz means on $\R^2$, the previous noncommutative theories may  not be sufficient and some new tools in harmonic analysis related to the geometry of Euclidean spaces should be brought in.

Our main new tool is  the noncommutative Kakeya maximal function.  Define the  Kakeya average operator by
\Bes
\mathcal{K}_R f(x)=|R|^{-1}\int_{R}f(x-y)dy,
\Ees
where $R$ is a rectangle centered at the origin with arbitrary orientation and eccentricity $N$ (see Section \ref{s:21Kak} for its definition).
As aforementioned, the study of the Kakeya maximal function $\sup_{R}|\mathcal{K}_R f(x)|$ is another important problem in harmonic analysis related to Bochner-Riesz means (see e.g. \cite{Wol99}).
Notice that the study of noncommutative Kakeya maximal functions is more difficult since it can not be defined directly.
It is easy to see that $\mathcal{K}_R$ could be dominated by the Hardy-Littlewood average operator with bound $N$. This implies that the noncommutative Kakeya maximal operator is $L_2$ bounded with norm $N^{\fr12}$.
In this paper, we shall establish its sharp $L_2$ norm---$\log N$ (see Theorem \ref{t:21Kak}), which is crucial to our study of noncommutative Bochner-Riesz means. To the best knowledge of the author, it is the first time that a sharp estimate of the noncommutative Kakeya maximal function is obtained in noncommutative analysis.
The proof here is quite technical and our strategy is the microlocal decomposition, together with theories of the Fourier transform and noncommutative square/maximal functions.

Below we sketch out the proof of the $L_p$ boundedness of the operator-valued Bochner-Riesz means on  $\R^2$ (i.e. Theorem \ref{t:21BR}).
To get the full boundedness of Bochner-Riesz means, by the duality and the noncommutative analytic interpolation theorem (see the appendix), it suffices to show the result for the case $p=4$.
We first make a dyadic decomposition: $B^\lam=\sum_{k}T_k$ and matters are reduced to proving the $L_4$ norm of $T_k$ has enough decay in $k$.
We next make a microlocal decomposition:  $T_k=\sum_lT_{k,l}$ where the support of each $T_{k,l}$ lies in a small piece (denoted by $\Ga_{k,l}$) of annulus with the major direction $e^{2\pi il2^{-\fr k2}}$. Notice that the $L_4$ norm of $T_k(f)$ has an expression
\Bes
\Big\|\sum_lT_{k,l}(f)\Big\|_{L_4(\N)}=\Big\|\sum_l\sum_{l'}T_{k,l}(f)^*T_{k,l'}(f)\Big\|_{L_2(\N)}^{\fr12}.
\Ees
If the major directions of these pieces are closed to each other (i.e. $|l-l'|\leq C$), then the above term is bounded by a column square function norm $\big\|\big(\sum_{l}|T_{k,l}(f)|^2\big)^{\fr12}\big\|_{L_4(\N)}$.
If $|l-l'|\geq C$, then we need a very important geometric observation: $\{\Ga_{k,l}-\Ga_{k,l'}\}_{l,l'}$ is finite overlapped. With this geometric estimate, the $L_4$ norm of $T_kf$ is bounded by a row square function norm $\big\|\big(\sum_{l}|T_{k,l}(f)^*|^2\big)^{\fr12}\big\|_{L_4(\N)}$.
Consequently we  use both column and row square function norms to control the $L_4$ norm of $T_k(f)$, which is consistent with the theory of noncommutative square functions.
To estimate this column/row square functions, we should do more analysis for the kernel of $T_{k,l}$. Roughly speaking, a key fact in our proof is that $T_{k,l}$ can be  bounded by  a sum of  Kakeya average operators $\mathcal{K}_R$s where orientations of $R$s are just in a fixed direction---the major direction of $\Ga_{k,l}$.
By the dual theory between $L_{p'}(\M;\ell_\infty)$ and $L_p(\M;\ell_1)$, our estimates for noncommutative square functions can be reduced to that of Kakeya maximal functions, which will be systematically studied in Section \ref{s:21Kak}.

The methods above heavily rely on the geometry of the plane. For the higher dimensional case, to get some nontrivial boundedness of Bochner-Riesz means, some more new tools in harmonic analysis should be transferred to the noncommutative setting. We hope to work this problem in the future.

This paper is organized as follows. First we give some preliminaries of noncommutative  $L_p$ spaces, noncommutative maximal/square functions and related lemmas in Section \ref{s:212}.
In Section \ref{s:21Kak}, we investigate the noncommutative Kakeya maximal function and establish its sharper estimate there.
In Section \ref{s:21BR}, we obtain the full $L_p$ boundedness of the operator-valued Bochner-Riesz means on $\R^2$. The proof is based on  sharper estimates of noncommutative Kakeya maximal functions and a square function inequality studied in the previous sections.
Section \ref{s:21appli} is devoted to the study of Bochner-Riesz means on quantum tori.
In this section, we first establish the full estimates of the operator-valued Bochner-Riesz means on usual tori $\T^2$ and then transfer this result to that on two-dimensional quantum tori (i.e. Theorem \ref{t:21BRqtorusi}).
Finally for the reader's convenience, we give a proof of the noncommutative analytic interpolation theorem in the appendix which may be known to experts.

\textbf{Notation}. Throughout this paper, the letter $C$ stands for a positive finite constant which is independent of the essential variables, not necessarily the same one in each occurrence. $A\lc B$ means $A\leq CB$ for some constant $C$. By the notation $C_\eps$  we mean that the constant depends on the parameter $\eps$. $A\approx B$ means that $A\lc B$ and $B\lc A$.
For any measurable set $A\subset\R^d$, we denote the Lebesgue measure by $|A|$.
$\Z_+$ denotes the set of all nonnegative integers and $\Z_+^d={\Z_+\times \cdots\times \Z_+}$ with $d$-tuples product. For $\alp\in\Z_+^d$ and $x\in\R^d$, $x^\alp=x_1^{\alp_1}\cdots x_d^{\alp_d}$. Set $\R_+=(0,\infty)$. $\forall s\in\R_+$, $\lfloor s \rfloor$ denotes the integer part of $s$.
We use LHS to represent left hand side of an expression.
Given a function $f$ on $\T^d$, the Fourier transform of $f$ is defined by $\wh{f}(k)=\int_{\T^d}f(x)e^{-2\pi i x\cdot k}dx$.
For a function $f$ on $\R^d$, define $\mathcal{F}f$ (or $\hat{f}$) and $\mathcal{F}^{-1}f$ (or $\check{f}$) the Fourier transform and the inversion Fourier transform of $f$  by
$$\mathcal{F}f(\xi)=\int_{\R^d} e^{-2\pi i\inn{x}{\xi}}f(x)dx,\ \ \ \ \mathcal{F}^{-1}f(\xi)=\int_{\R^d}e^{2\pi i\inn{x}{\xi}}{f(x)dx}.$$
\vskip0.24cm

\section{Preliminaries and some lemmas}\label{s:212}
In this section, we introduce some basic knowledge of noncommutative harmonic analysis including noncommutative $L_p$ spaces, maximal functions, square functions and many operator-valued inequalities which are useful in this paper.
\subsection{Noncommutative $L_{p}$-spaces}\label{s:212lp}
Let $\M$ be a semifinite von Neumann algebra equipped with a normal semifinite faithful (\emph{n.s.f.} in short) trace $\tau$.
Denote by $\M_{+}$ the positive part of $\M$ and let $\mathcal{S_{+}}$  be the set of all $x\in\M_{+}$ whose support projections have finite trace. Let $\mathcal{S}$ be the linear span of $\mathcal{S_{+}}$, then $\mathcal{S}$ is a weak $*$ dense $\ast$-subalgebra of $\M$.
Consider $0< p<\infty$. For any $x\in\mathcal{S}$, $|x|^{p}\in\mathcal{S}$ and we set
$$\|x\|_{L_p(\M)}=\big(\tau(|x|^p)\big)^{1/p},\ \ x\in\mathcal{S},$$
where $|x|=(x^{\ast}x)^{\frac{1}{2}}$ is the modulus of $x$. Define the noncommutative $L_{p}$ space associated with $(\M,\tau)$ by the completion of $(\mathcal{S},\|\cdot \|_{L_p(\M)})$ and set it as $L_{p}(\M)$. For convenience, if $p=\infty$, we define $L_{\infty}(\M) = \M$ equipped with the operator norm $\|\cdot \|_{\M}$. Let $L_{p}^+(\M)$ denote the positive part of $L_{p}(\M)$.
A lot of basic properties of classical $L_p$ spaces, such as Minkowski's inequality, H\"older's inequality, dual property, real and complex interpolation,  have been transferred to this noncommutative setting.
In particular, the following monotone properties are frequently used in this paper: for $a,b\in\M$ and $\alp\in\R_+$,
\Be\label{e:21mono}
\|a\|_{L_p(\M)}\leq\|b\|_{L_p(\M)},\text{ if } 0\leq a\leq b;
\Ee
\Be\label{e:21mono01}
a^\alp\leq b^\alp, \text{ if $0\leq a\leq b$ and $0<\alp<1$.}
\Ee
For more about noncommutative $L_p$ spaces, we refer to the very detailed introduction in the survey article \cite{PX03} or the book \cite{Xu07}.

In this paper, we are interested in the noncommutative $L_p$ space on the tensor von Neumann algebra $\N=L_\infty(\R^d)\tens\M$. Set the tensor trace $\vphi=\int_{\R^d}dx\otimes\tau$. Define the noncommutative space $L_p(\N)$ associated with pairs $(\N,\vphi)$. Notice that $L_p(\N)$ is isometric to $L_p(\R^d; L_p(\M))$ the Bochner $L_p$ space on $\R^d$ with values in $L_p(\M)$.

\subsection{Noncommutative maximal functions}\label{s:212max}
It is difficult to define a noncommutative maximal function straightforwardly since two general elements in a von Neumann algebra may not be comparable. This obstacle can be overcome by defining the maximal norm directly.
We adopt the definition of the noncommutative maximal norm
introduced by G. Pisier \cite{Pis98} and M. Junge \cite{Jun02}.
\begin{definition}[$L_p(\M;\ell_\infty)$]
We define
$L_p(\M;\ell_\infty)$ the space of all sequences
$x=\{x_n\}_{n\in\Z}$ in $L_p(\M)$ which admits a factorization of the
following form: there exist $a, b\in L_{2p}(\M)$ and a bounded
sequence $y=\{y_n\}_{n\in\Z}$ in $L_\infty(\M)$ such that
 $x_n=ay_nb$, $\forall\; n\in\Z$.
The norm of  $x$ in $L_p(\M;\ell_\infty)$ is given by
 $$\|x\|_{L_p(\M;\ell_\infty)} =\inf\big\{\|a\|_{L_{2p}(\M)}\,
 \sup_{n\in\Z}\|y_n\|_{L_\infty(\M)}\,\|b\|_{L_{2p}(\M)}\big\} ,$$
where the infimum is taken over all factorizations of $x$ as above.
\end{definition}

If $x=\{x_n\}_{n\in\Z}$ is a sequence of positive elements, then $x\in L_p(\M;\ell_\infty)$ if and only if there exists a positive element $a\in L_p(\M)$ such that $0<x_n\leq a$, and
\Be\label{e:21positivem}
\|x\|_{L_p(\M;\ell_\infty)}=\inf\{\|a\|_{L_p(\M)}:\ 0<x_n\leq a, \forall n\in\Z\}.
\Ee

Similarly if $x=\{x_n\}_{n\in\Z}$ is a sequence of self-adjoint elements, then $x\in L_p(\M;\ell_\infty)$ if and only if there exists a positive element $a\in L_p(\M)$ such that $-a\leq x_n\leq a$, and
\Be\label{e:21selfadjointm}
\|x\|_{L_p(\M;\ell_\infty)}=\inf\{\|a\|_{L_p(\M)}:\  -a\leq x_n\leq a, \forall n\in\Z\}.
\Ee

More generally, if $\Lambda$ is an index set, we define  $L_p (\M;
\ell_{\infty}(\Lambda))$ as the space of all $x =
\{x_{\lambda}\}_{\lambda \in \Lambda}$ in $L_p (\M)$ that can be
factorized as
 $$x_{\lambda}=ay_{\lambda} b\quad\mbox{with}\quad a, b\in L_{2p}(\M),\; y_{\lambda}\in L_\infty(\M),\; \sup_{\lambda}\|y_{\lambda}\|_{L_\infty(\M)}<\infty.$$
Then the norm of $L_p (\M; \ell_{\infty}(\Lambda))$ is defined by
$$\|x\|_{L_p(\M;\ell_\infty(\Lam))}=\inf_{x_{\lambda}=ay_{\lambda} b}\big\{\|a\|_{L_{2p}(\M)}\,
\sup_{{\lambda}\in\Lambda}\|y_{\lambda}\|_{L_\infty(\M)}\,\|b\|_{L_{2p}(\M)}\big\}.$$
It was shown in \cite{JX07} that $x\in L_p (\M;
\ell_{\infty}(\Lambda))$ if and only if $\sup\big\{\|x\|_{L_p(\M;\ell_\infty(J))}\;:\; J\subset\Lambda,\; J\textrm{ is finite}\big\}<\infty$ and moreover in this case, the norm $\|x\|_{L_p(\M;\ell_\infty(\Lam))}$ is equal to the above supremum.

If $x=\{x_\lam\}_{\lam\in\Lambda}$ is positive (resp. self-adjoint),  $\|x\|_{L_p(\M;\ell(\Lambda))}$ has the similar property of \eqref{e:21positivem} (resp. \eqref{e:21selfadjointm}).

We will often use $\|\sup\limits_{\lam\in\Lambda}x_\lam\|_{L_p(\M)}$ to represent $\|x\|_{L_p(\M;\ell(\Lambda))}$.
However we point out that $\|\sup\limits_{\lam\in\Lambda}x_\lam\|_{L_p(\M)}$ is just a notation since $\sup\limits_{\lam\in\Lambda}x_\lam$ makes no sense in the noncommutative setting.

To study the dual property of the above spaces $L_p(\M;\ell_\infty)$, we need to introduce another space.

\begin{definition}[$L_p(\M;\ell_1)$]
Define $L_p(\M;\ell_1)$ as the space of all sequences $\{y_n\}$ in $L_p(\M)$ which could be factorized as
\Bes
y_n=\sum_ku^*_{k,n}v_{k,n},\quad \forall n\in\Z,
\Ees
for two families $\{u_{k,n}\}_{k,n\in\Z}$ and $\{v_{k,n}\}_{k,n\in\Z}$ in $L_{2p}(\M)$ such that
$\sum_{k,n\in\Z}u^*_{k,n}u_{k,n}\in L_p(\M)$ and $\sum_{k,n\in\Z}v^*_{k,n}v_{k,n}\in L_p(\M)$.
$L_p(\M;\ell_1)$ is equipped with the norm
\Bes
\|\{y_n\}\|_{L_p(\M;\ell_1)}=\inf\Big\|\sum_{k,n\in\Z}u^*_{k,n}u_{k,n}\Big\|_{L_p(\M)}^{\fr12}\Big\|\sum_{k,n\in\Z}v^*_{k,n}v_{k,n}\Big\|_{L_p(\M)}^{\fr12},
\Ees
where the infimun is taken over all decompositions of $\{y_n\}$ as above.
\end{definition}
It is not difficult to see that  if $y_n\geq0$ for all $n\in\Z$, $\{y_n\}\in L_p(\M;\ell_1)$ if and only if $\sum_{n\in\Z}y_n\in L_p(\M)$ (see e.g. \cite{Xu07}). In such a case, we have the following equality
\Bes
\|\{y_n\}\|_{L_p(\M;\ell_1)}=\Big\|\sum_{n\in\Z}y_n\Big\|_{L_p(\M)}.
\Ees

We introduce the following basic duality theorem of $L_p(\M;\ell_1)$, which has been established by M. Junge and Q. Xu in \cite{JX07}.
\begin{lemma}\label{l:21dual}
\rm{(i).} Suppose $1\leq p<\infty$. Let $p'$ be the conjugate index: $\fr 1p+\fr{1}{p'}=1$. Then the dual space of $L_p(\M;\ell_1)$ is $L_{p'}(\M;\ell_\infty)$. The element $x=\{x_n\}\in L_{p'}(\M;\ell_\infty)$ acts on $L_p(\M;\ell_1)$ as follows
\Bes
\inn{x}{y}=\sum_{n\in\Z}\tau(x_ny_n),\quad \forall y=\{y_n\}\in L_p(\M;\ell_1).
\Ees
\rm{(ii).} Suppose $1\leq p\leq\infty$. For any $x\in L_p(\M;\ell_\infty)$, we have
\Bes
\|\{x_n\}\|_{L_p(\M;\ell_\infty)}=\sup\Big\{\sum_{n}\tau(x_ny_n): y=\{y_n\}\in L_{p'}(\M;\ell_1)\ \text{and}\ \|y\|_{L_{p'}(\M;\ell_1)}\leq1\Big\}.
\Ees
Moreover if $x$ is positive, then
\Bes
\|\{x_n\}\|_{L_p(\M;\ell_\infty)}=\sup\Big\{\sum_{n}\tau(x_ny_n): y_n\in L^+_{p'}(\M)\ \text{and}\ \|\sum_ny_n\|_{L_{p'}(\M)}\leq1\Big\}.
\Ees
\end{lemma}

\subsection{Noncommutative square functions}\label{s:212squa}\quad
To define the noncommutative square function, we should first introduce the so-called column and row function spaces. Let  $f=\{f_{j}\}$ be a finite sequence in $L_{p}(\M)$ where $1\leq p\leq\infty$.
Define
$$\|\{f_{j}\}\|_{L_{p}(\M; \ell_{2}^{r})}=\big\|(\sum|f^{\ast}_{j}|^{2})^{\frac{1}{2}}\big\|_{L_p(\M)},
\
\|\{f_{j}\}\|_{L_{p}(\M; \ell_{2}^{c})}=\big\|(\sum|f_{j}|^{2})^{\frac{1}{2}}\big\|_{L_p(\M)}.$$

\begin{definition}[$L_p(\M;\ell_2^{rc})$]
We define the spaces $L_p(\M;
\ell_{2}^{rc})$ as follows:
\begin{enumerate}[\rm (i).]
\item If $p\geq2$, $L_p(\M;
\ell_{2}^{rc})=L_{p}(\M; \ell_{2}^{c})\cap L_{p}(\M; \ell_{2}^{r})$
equipped with the intersection norm:
$$\|\{f_{j}\}\|_{L_p(\M;
\ell_{2}^{rc})}=\max\{\|\{f_{j}\}\|_{L_{p}(\M; \ell_{2}^{c})},
\|\{f_{j}\}\|_{L_{p}(\M; \ell_{2}^{r})}\}.$$
\item If $p<2$, $L_p(\M;
\ell_{2}^{rc})=L_{p}(\M; \ell_{2}^{c})+ L_{p}(\M; \ell_{2}^{r})$
equipped with the sum norm:
$$\|\{f_{j}\}\|_{L_p(\M;
\ell_{2}^{rc})}=\inf\{\|\{g_{j}\}\|_{L_{p}(\M; \ell_{2}^{c})}+
\|\{h_{j}\}\|_{L_{p}(\M; \ell_{2}^{r})}\},$$
where the infimun is taken over all decompositions $f_{j}=g_{j}+h_{j}$ with $g_{j}$ and $h_{j}$ in $L_{p}(\M)$.
\end{enumerate}
\end{definition}

It is easy to see that
$L_2(\M; \ell_{2}^{r}) = L_2(\M;
\ell_{2}^{c})=L_2(\M;
\ell_{2}^{rc})$.
Next we introduce some inequalities for $L_p(\M;\ell_2^{rc})$. The first one is H\"older type inequality whose proof can be found in  \cite{Xu07}.

\begin{lemma}\label{l:21vecth}
Let $0<p,q,r\leq\infty$ be such that $1/r=1/p+1/q$. Then for any $f\in L_p(\M;\ell^c_2)$ and $g\in L_q(\M;\ell^c_{2})$,
$$\Big\|\sum_if_i^*g_i\Big\|_{L_r(\M)}\leq \Big\|\Big(\sum_i |f_i|^2\Big)^{\frac12}\Big\|_{L_p(\M)}
\Big\|\Big(\sum_i|g_i|^2\Big)^{\frac12}\Big\|_{L_q(\M)}.$$
\end{lemma}

The second one is the noncommutative Khintchine inequality for the Rademacher sequence as follows.
\begin{lemma}[see \cite{Xu07}]\label{l:21Kin}
Let $1\leq p<\infty$ and $\{x_n\}$ be a finite sequence in $L_p(\M)$. Then
\Bes
\big\|\sum_nx_n\eps_n\big\|_{L_p(\Om;L_p(\M))}\approx\|\{x_n\}\|_{L_p(\M;\ell_2^{rc})},
\Ees
where $\{\eps_n\}$ is a Rademacher sequence on a probability space $(\Om,P)$.
\end{lemma}

We also require some convexity inequalities for the operator-valued function in this paper.
The following one is Cauchy-Schwarz type inequality which can found in \cite[Page 9]{Mei07}).
\begin{lemma}
Let $(\Sigma,\mu)$ be a measure space. Suppose that $f:\ \Sigma\rta\M$ is a weak-$*$ integrable function and $g:\Sigma\rta\mathbb C$ is an integrable function. Then
\Be\label{e:21conv}
\Big|\int_{\Sigma}f(x)g(x)d\mu(x)\Big|^2\leq\int_{\Sigma}|f(x)|^2d\mu(x)\int_{\Sigma}|g(x)|^2d\mu(x),
\Ee
where $\leq$ is understood as the partial order in the positive cone of $\M$.
\end{lemma}

Finally we introduce the following vector-valued Plancherel theorem which will be mostly used in  square function estimates: Let $\N=L_\infty(\R^d)\tens\M$, then we have
\Be\label{e:21planch}
\|\mathcal{F} f\|_{L_2(\N)}=\|f\|_{L_2(\N)},
\Ee
which is a consequence of the fact $L_2(\M)$ is a Hilbert space. Such vector-valued Plancherel theorem is sufficient in most part of our proof though sometimes we need a more general operator-valued Parseval's relation:
for $f,g\in L_2(L_\infty(\R^d)\tens\M)$, we have
\Be\label{e:21pars}
\int_{\R^d}g^*(x)f(x)dx=\int_{\R^d}(\mathcal{F}g)^*(\xi)\mathcal{F}(f)(\xi)d\xi.
\Ee
\vskip0.24cm

\section{Noncommutative Kakeya maximal functions}\label{s:21Kak}

In this section, we study the boundedness of noncommutative Kakeya maximal functions. Before that we give several definitions and lemmas.
We only consider $d=2$. Set $\N=L_\infty(\R^2)\tens\M$ and $\vphi=\int_{\R^2}dx\otimes\tau$. Denote $L_p(\N)$ the noncommutative $L_p$ space associated with pairs $(\N,\vphi)$.

The main preliminaries of noncommutative maximal functions have been given in Subsection \ref{s:212max}.
We introduce the noncommutative Kakeya maximal function as follows.
Define the eccentricity of a rectangle by the ratio of the length of its long side to that of its short side.
Let $N$ be a positive integer.
Define the set $\mathcal{R}_N$ as rectangles in the plane of arbitrary orientation whose center is the origin and eccentricity  is $N$.
For $f\in L_p(\N)$, we define the Kakeya average operator as follows
\Be\label{d:21Kak}
\mathcal{K}_R f(x)=|R|^{-1}\int_{R}f(x-y)dy,
\Ee
where $R$ is a rectangle belonging to $\mathcal{R}_N$.
We are mainly interested in the maximal $L_2$ norm of noncommutative Kakeya average operator, since it is a crucial estimate in the study of noncommutative Bochner-Riesz means.

Let us first give a trivial bound of its $L_2$ norm.
Recall that the Hardy-Littlewood average operator is defined by
\Bes
M_Q f(x)=\fr{1}{|Q|}\int_Qf(x-y)dy,
\Ees
where $Q$ is a cube in $\R^2$ with center zero and arbitrary orientation.
For any rectangle $R\in \mathcal{R}_N$, there exists a cube $Q$ such that $R\subset Q$ and $l(Q)$ equals to the  length of long side of $R$.
Consider $f$ as a positive function in $\N$.
Then $\mathcal{K}_Rf(x)\leq NM_Q f(x)$. Using the noncommutative Hardy-Littlewood maximal operator is of weak type $(1,1)$ (see \cite{Mei07}), we get the Kakeya maximal operator is of weak type $(1,1)$ with bound $N$.
Applying the noncommutative Marcinkiewicz interpolation theorem in \cite{JX07}, together with the fact that the maximal operator of $\mathcal{K}_R$ is of $(\infty,\infty)$, we get that the maximal operator of $\mathcal{K}_R$ is of strong $(2,2)$ type with bound $N^{\fr12}$. However this bound is quite rough and not sufficient for our later application.
The following improved bound is our main result in this section.
\begin{theorem}\label{t:21Kak}
Let $\mathcal{K}_R$ be the Kakeya average operator defined in \eqref{d:21Kak}. Then for $f\in L_2(\N)$, we have
\Bes
\|\sup_{R\in \mathcal{R}_N} \mathcal{K}_R f\|_{L_2(\N)}\lc(\log N)\|f\|_{L_2(\N)}.
\Ees
\end{theorem}
Combining the noncommutative Marcinkiewicz interpolation theorem (see \cite{JX07}), together with a trivial weak type $(1,1)$ bound  and a strong $(\infty,\infty)$ bound of $\mathcal{K}_R$, we immediately get the following corollary.
\begin{corollary}
Let $\mathcal{K}_R$ be defined in \eqref{d:21Kak}. Then for any $1<p<\infty$, we have
\Bes
\|\sup_{R\in \mathcal{R}_N} \mathcal{K}_R\|_{L_p(\N)\rta L_p(\N)}\lc
\begin{cases}
N^{\fr 2p-1}(\log N)^{\fr{2}{p'}},\,&\hbox{if}\,\,\ 1<p<2;\\
(\log N)^{\fr 2p},\,&\hbox{if}\,\,\ 2\leq p<\infty.
\end{cases}
\Ees
\end{corollary}

It should be pointed out that the bound $\log N$ in Theorem \ref{t:21Kak} is sharp even in the commutative case,
see \cite[Proposition 5.3.4]{Gra14M}.
A. C\'ordoba \cite{Cor77} first obtained the boundedness of the Kakeya maximal function on $L_2(\R^2)$ with norm $(\log N)^2$.
The sharp bound $\log N$ was later established by J. O. Str\"omberg \cite{Str78} where he used several estimates of distribution functions and some geometric constructions.
S. Wainger \cite{Wai79} also obtained the sharp bound $\log N$ without a proof but he mentioned that the idea given by A. Nagel, E. M. Stein, S. Wainger \cite{NSW78} can be modified to his setting.
It is well-known that the distribution function is difficult to deal with in the noncommutative setting (for example the weak $(1,1)$ boundedness problem is a challenge problem).
Hence the method from J. O. Str\"omberg \cite{Str78} may be difficult to be applied in the noncommutative setting. The strategies used in our proof below are the Fourier transform, square and maximal function theories and the microlocal decomposition, which are mainly motivated by  \cite{NSW78} and  \cite{Wai79}.

Before giving the proof of Theorem \ref{t:21Kak}, we introduce the noncommutative   directional Hardy-Littlewood average operator defined by
\Be\label{e:21Mdirec}
M^e_{h}(f)(x)=\fr{1}{2h}\int_{-h}^{h}f(x-ey)dy,
\Ee
where $e$ is a unit vector in $\R^2$. By using the standard method of rotation, the definition of maximal norm in \eqref{e:21positivem} and the fact one dimensional noncommutative Hardy-Littlewood maximal operator is of strong type $(p,p)$ for $1<p\leq\infty$ (see \cite{Mei07}), the author and his collaborators recently established  the following result in \cite[Lemma 6.3]{HLX20}.

\begin{lemma}\label{l:21Mdirec}
Let $e$ be a unit vector. Define $M_h^e$ in \eqref{e:21Mdirec}.
Let $1<p\leq\infty$. Then we have
\Bes
\|\sup_{h>0}M_h^e(f)\|_{L_p(\N)}\lc\|f\|_{L_p(\N)}.
\Ees
\end{lemma}

Now we are in a position to prove Theorem \ref{t:21Kak}.
\begin{proof}[Proof of Theorem \ref{t:21Kak}]
Let us start with several reductions.
Without loss of generality, we suppose that $f$ is positive since the general case just follows by decomposing $f$ as linear combination of four positive functions.

To prove our estimate, it suffices to consider the case $N=2^m$. In fact, assume that we show this theorem for $N=2^m$, we can prove the general case for arbitrary $N$ as follows.
For any positive integer $N$, there exists a positive integer $m$ such that $2^{m-1}<N\leq 2^m$, i.e. $m-1<\log N\leq m$.
Then for any $R\in \mathcal{R}_N$, by enlarging the long side of $R$ such that its eccentricity increases to $2^m$, we get a new rectangle $\tilde{R}$. Then
$
\mathcal{K}_Rf(x)\leq \fr{2^m}{N}\mathcal{K}_{\tilde{R}}f(x).
$
Therefore we get
\Bes
\|\sup_{R\in \mathcal{R}_N}\mathcal{K}_Rf\|_{L_2(\N)}\leq \fr{2^m}{N}\|\sup_{R\in \mathcal{R}_{2^m}}\mathcal{K}_Rf\|_{L_2(\N)}\lc m\|f\|_{L_2(\N)}\approx(\log N)\|f\|_{L_2(\N)}.
\Ees

Similar to that in the commutative case, by symmetry and rotation, we only need to control the average over those rectangles that have eccentricity $N$, but whose major axes make angles $\tet_k$ with the $x$-axis such that $\tan\tet_k=k/N, k=0,\cdots,N-1$ (see e.g. Section 3.11, Chapter X in \cite{Ste93}).
We abuse notation and still define $\mathcal{R}_N$ as those preceding rectangles.

After these reductions, below we smooth the average operator. Set $u^N_k=(N,k),k=0,\cdots,N-1$ and let $e_2=(0,1)\in\R^2$.
Choose a nonnegative, radially decreasing and smooth function $\psi$ such that $\psi(x)=1$ if $|x|<\fr12$ and $\supp\psi\subset\{|x|<2\}$.
Define $\psi_h(t)=h^{-1}\psi(t/h)$.
To prove our theorem, it is sufficient to consider the average operator
\Bes
A_{h}^{k,N}(f)(x)=\int_{\R}\int_{\R}f(x-u^N_kt-e_2s)\psi_{h}(t)\psi_{h}(s)dsdt,\ h>0.
\Ees
Indeed, since $f\geq0$, by some elementary geometric observation, for any $R\in \mathcal{R}_N$ with the major direction $e^{i\tet}$ where $\tan\tet=\fr kN$ for some $k=0,\cdots,N-1$, there exists $h$ such that
\Bes
\mathcal{K}_Rf(x)\leq CA_{h}^{k,N}(f)(x)
\Ees
where the constant $C$ is independent of $k,N,f$.
On the other hand, for $0\leq k\leq N-1$, $h\in\R_+$, there  exists $R\in \mathcal{R}_N$ with the major direction $e^{i\tet}$ such that $\tan\tet=\fr kN$,
\Bes
A_h^{k,N}(f)(x)\leq C \mathcal{K}_R f(x).
\Ees
Therefore by the definition of positive maximal norm in \eqref{e:21positivem}, we get
\Bes
\Big\|\sup_{R\in \mathcal{R}_N} \mathcal{K}_R f\Big\|_{L_2(\N)}{\approx}\Big\|\sup_{0\leq k<N \atop h>0}A_h^{k,N}(f)\Big\|_{L_2(\N)}.
\Ees

Notice that there are two averages for two different directions in the operator $A_h^{k,N}$.
Our next goal is to reduce it to one average with the help of the directional Hardy-Littlewood maximal operator.
By our choice of $\psi$, it is easy to see that
\Bes
A_h^{k,N}(f)(x)\lc M^{e_2}_{2h}[M_{h}^{k,N}(f)](x)
\Ees
where $M_{2h}^{e_2}$ is the directional Hardy-Littlewood average operator defined in \eqref{e:21Mdirec} and $M_{h}^{k,N}(f)$ is defined as
\Bes
M_h^{k,N}(f)(x)=\int_{\R}f(x-u^N_kt)\psi_{h}(t)dt.
\Ees
Therefore by the $L_2$ boundedness of maximal operator of $M^e_h$ in Lemma \ref{l:21Mdirec}, to prove our theorem, it suffices to show
\Be\label{e:21MhkN}
\Big\|\sup_{0\leq k<N\atop h>0}M_h^{k,N}(f)\Big\|_{L_2(\N)}\lc (\log N)\|f\|_{L_2(\N)}.
\Ee

Next we reduce the study of $M_{h}^{k,N}(f)(x)$ to its lacunary case $M_{2^j}^{k,N}(f)(x)$. For any $h>0$, there exists $j\in\Z$ such $2^{j-1}\leq h<2^{j}$. Then we get
\Bes
M_{h}^{k,N}(f)(x)=\fr1h\int_{\R}f(x-u^N_kt)\psi(\fr th)dt\lc\fr{1}{2^j}\int_{\R}f(x-u^N_kt)\psi(\fr{t}{2^j})dt=M_{2^j}^{k,N}(f)(x),
\Ees
where the second inequality just follows from the radially decreasing property of $\psi$.
Hence we only need to consider the lacunary operator $M_{2^j}^{k,N}(f)(x)$.
At present time, we conclude that the proof of our main theorem is reduced to show
\Be\label{e:21mdyad}
\Big\|\sup_{0\leq k<N\atop j\in\Z}M_{2^j} ^{k,N}(f)\Big\|_{L_2(\N)}\lc(\log N)\|f\|_{L_2(\N)}.
\Ee

Let $N=2^m$. In the following, we will establish a key inequality
\Be\label{e:21Kakeya}
\Big\|\sup_{0\leq k<2^m\atop j\in\Z} M_{2^j}^{k,2^m}(f)\Big\|_{L_2(\N)}\leq C\|f\|_{L_2(\N)}+\Big\|\sup_{0\leq k<2^{m-1}\atop j\in\Z} M_{2^j}^{k,2^{m-1}}(f)\Big\|_{L_2(\N)}
\Ee
with the constant $C$ independent of $m$. Notice that when $m=1$, $M_{2^j}^{0,1}(f)\leq M^{e_1}_{2^{j+1}}(f)$, where $e_1=(1,0)$ and $M_h^e$ is the noncommutative directional Hardy-Littlewood average operator in \eqref{e:21Mdirec}.
Using the $L_2$ boundedness of maximal $M_h^e$ in Lemma \ref{l:21Mdirec}, we get
$$\Big\|\sup_{j\in\Z} M_{2^j}^{0,1}(f)\Big\|_{L_2(\N)}\lc\|f\|_{L_2(\N)}.$$
Then it is easy to see that the required estimate \eqref{e:21mdyad} just follows from \eqref{e:21Kakeya} with an induction argument.

The rest of this section is devoted to the proof of \eqref{e:21Kakeya}. For convenience, set $e_{m}^l=(1,\fr{l}{2^{m}})$.  Define $\Ga_l=\{\xi:|\inn{\fr{\xi}{|\xi|}}{e_m^l}|\leq \fr{c}{2^m}\}$ where $c$ is a constant independent of $m$ such that $\Ga_1, \Ga_2, \cdots, \Ga_{2^{m-1}-1}$ are disjoint from each other.
Notice that $\Ga_l$s are equally distributed in $\{(x,y): |x|\leq |y|,x\leq0,y\geq0\}$ or $\{(x,y): |x|\leq |y|, x\geq0, y\leq0\}$ (see Figure \ref{f:21max} later). We split $f$ as $f_l$ and $r_l$ which are defined by
\Bes
\widehat{f_l}(\xi)=\chi_{\Ga_{2l+1}\cup\Ga_{2l}}(\xi)\widehat{f}(\xi),\quad
\widehat{r_l}(\xi)=\chi_{\Ga_{2l+1}^c\cap\Ga^c_{2l}}(\xi)\widehat{f}(\xi).
\Ees

Recall that $\vphi=\int_{\R^2}dx\otimes\tau$. Since $M_{2^j}^{k,N}(f)(x)$ is positive in $L_2(\N)$, by the duality (see (ii) in Lemma \ref{l:21dual}), there exists a positive sequence $\{h_{2^j}^k\}\in L_2(\N;\ell_1)$ with norm $\|\sum\limits_{0\leq k<2^m\atop j\in\Z}h_{2^j}^k\|_{L_2(\N)}\leq 1$ such that
\Bes
\begin{split}
\Big\|\sup_{0\leq k<2^m\atop j\in\Z} M_{2^j}^{k,2^m}(f)\Big\|_{L_2(\N)}&=\vphi\Big(\sum_{0\leq k<2^m\atop j\in\Z}M_{2^j}^{k,2^m}(f)h_{2^j}^k\Big)\leq I+II
\end{split}
\Ees
where
$$I=\Big|\vphi\Big(\sum_{0\leq k<2^m\atop j\in\Z}[M_{2^{j+1}}^{\lfloor\fr k2\rfloor,2^{m-1}}(f)]h_{2^j}^k\Big)\Big|,$$
$$II=\Big|\vphi\Big(\sum_{0\leq k<2^m \atop j\in\Z}[M_{2^j}^{k,2^m}(f)-M_{2^{j+1}}^{\lfloor \fr k2\rfloor,2^{m-1}}(f)]h_{2^j}^k\Big)\Big|.$$

We consider the first term $I$. Notice that for every $0\leq k<2^{m-1}$, we have
\Be\label{e:21evenodd}
M_{2^{j+1}}^{\lfloor\fr{2k}{2}\rfloor,2^{m-1}}(f)=M_{2^{j+1}}^{\lfloor\fr{2k+1}{2}\rfloor,2^{m-1}}(f).
\Ee
Recall that $M_{2^{j+1}}^{k,2^{m-1}}(f)$ is positive in $\N$. Then by the duality (see (ii) in Lemma \ref{l:21dual}), the definition of $L_2(\N;\ell_\infty)$  in \eqref{e:21positivem} and the preceding equality \eqref{e:21evenodd}, we get
\Bes
\begin{split}
I&\leq\Big\|\sup_{0\leq k<2^m\atop j\in\Z}M_{2^{j+1}}^{\lfloor\fr k2\rfloor,2^{m-1}}(f)\Big\|_{L_2(\N)}=\Big\|\sup_{0\leq k<2^{m-1}\atop j\in\Z}M_{2^{j+1}}^{k,2^{m-1}}(f)\Big\|_{L_2(\N)}\\
&=\Big\|\sup_{0\leq k<2^{m-1}\atop j\in\Z}M_{2^{j}}^{k,2^{m-1}}(f)\Big\|_{L_2(\N)},
\end{split}
\Ees
which is exact the second term in right side of \eqref{e:21Kakeya}.

Now we turn to $II$. To finish the proof of \eqref{e:21Kakeya}, we only need to show that $II$ is controlled by $\|f\|_{L_2(\N)}$.
By making a dilation, it is easy to check that for $0\leq l<2^{m-1}$,
\Be\label{e:21dia}
M_{2^{j+1}}^{l,2^{m-1}}(f)(x)=M_{2^j}^{2l,2^m}(f)(x).
\Ee
Therefore we see that the terms related to even $k$ in the sum of $II$ equal to zero. Applying \eqref{e:21evenodd} and \eqref{e:21dia} again, we rewrite  the odd terms in $II$  as follows
\Bes
\begin{split}
II&=\Big|\vphi\Big(\sum_{0\leq l<2^{m-1}\atop j\in\Z}[M_{2^j}^{2l+1,2^m}(f)-M_{2^j}^{2l,2^m}(f)]h_{2^j}^{2l+1}\Big)\Big|\\
&\leq \Big|\vphi\Big(\sum_{0\leq l<2^{m-1}\atop j\in\Z}[M_{2^j}^{2l+1,2^m}(f_l)h_{2^j}^{2l+1}]\Big)\Big|+\Big|\vphi\Big(\sum_{0\leq l<2^{m-1}\atop j\in\Z}[M_{2^j}^{2l,2^m}(f_l)h_{2^j}^{2l+1}]\Big)\Big|\\
&\quad+\Big|\vphi\Big(\sum_{0\leq l<2^{m-1}\atop j\in\Z}[M_{2^j}^{2l+1,2^m}(r_l)-M_{2^j}^{2l,2^m}(r_l)]h_{2^j}^{2l+1}\Big)\Big|
=:II_1+II_2+II_3.
\end{split}\Ees

Let us consider $II_1$ firstly. Using the duality in (ii) of Lemma \ref{l:21dual}, we get
\Bes
II_1\leq
\big\|\sup_{0\leq l<2^{m-1}\atop j\in\Z}M_{2^j}^{2l+1,2^m}(f_l)\big\|_{L_2(\N)}.
\Ees
Recall $e_m^l=(1,\fr{l}{2^m})$. Let $\tilde{e}_{m}^{2l+1}$ be the unit vector in the direction along $e_{m}^{2l+1}$, i.e.  $\tilde{e}_{m}^{2l+1}={e}_{m}^{2l+1}/|{e}_{m}^{2l+1}|$.
Then it is straightforward to verify that for a positive function $g$,
\Bes
M_{2^j}^{2l+1,2^m}(g)(x)\lc M^{\tilde{e}_{m}^{2l+1}}_{2^{j+m}}(g)(x),
\Ees
where the right side of the above inequality is the directional Hardy-Littlewood average operator defined in \eqref{e:21Mdirec}.
Now using the  $L_2$ boundedness of maximal operator of $M^{e}_{h}$ in Lemma \ref{l:21Mdirec}, we get for any positive function $g\in L_2(\N)$,
\Be\label{e:21B2jl}
\big\|\sup_{j\in\Z}M_{2^j}^{2l+1,2^m}(g)\big\|_{L_2(\N)}\lc\|g\|_{L_2(\N)}.
\Ee
Consequently \eqref{e:21B2jl} holds for any $g\in L_2(\N)$ by decomposing it as linear combination of four positive functions.

Observe that $M_{2^j}^{2l+1,2^m}(f_l)$ may not be positive. So we can not apply the maximal norm in \eqref{e:21positivem}. Recall that for any $a\in\N$, we can decompose it as linear combination of two self-adjoint elements:
\Bes
a=\Re(a)+i\Im(a),\quad\text{where}\quad \Re(a)=\fr12{(a+a^*)},\ \Im(a)=\fr{1}{2i}(a-a^*).
\Ees
Hence we can write $M_{2^j}^{2l+1,2^m}(f_l)=\Re[M_{2^j}^{2l+1,2^m}(f_l)]+i\Im[M_{2^j}^{2l+1,2^m}(f_l)]$. Then utilizing Minkowski's inequality, we get
\Bes
II_1\leq
\big\|\sup_{0\leq l<2^{m-1}\atop j\in\Z}\Re[M_{2^j}^{2l+1,2^m}(f_l)]\big\|_{L_2(\N)}+
\big\|\sup_{0\leq l<2^{m-1}\atop j\in\Z}\Im[M_{2^j}^{2l+1,2^m}(f_l)]\big\|_{L_2(\N)}.
\Ees

We first consider the real part. Notice that $M_{2^j}^{2l+1,2^m}(f_l)^*=M_{2^j}^{2l+1,2^m}(f_l^*)$. Then by Minkowski's inequality and \eqref{e:21B2jl}, we get
\Bes
\begin{split}
\big\|\sup_{ j\in\Z}\Re[M_{2^j}^{2l+1,2^m}(f_l)]\big\|_{L_2(\N)}\leq&
\fr12\big\|\sup_{j\in\Z} M_{2^j}^{2l+1,2^m}(f_l)\big\|_{L_2(\N)}\\
&+\fr12
\big\|\sup_{ j\in\Z}M_{2^j}^{2l+1,2^m}(f_l^*)\big\|_{L_2(\N)}\lc\|f_l\|_{L_2(\N)}.
\end{split}
\Ees
Rewrite this estimate via the equivalent definition of $L_p(\M;\ell_\infty)$ in \eqref{e:21selfadjointm}, we obtain that there exists $F_l>0$ such that for each $j\in\Z$,
$$-F_l\leq \Re[M_{2^j}^{2l+1,2^m}(f_l)]\leq F_l \quad\text{and}\quad \|F_l\|_{L_2(\N)}\lc \|f_l\|_{L_2(\N)}.$$
Then by setting $F=(\sum_lF_l^2)^{\fr12}$, we see that for each $0\leq l<2^{m-1},j\in\Z$, we have $-F\leq \Re[M_{2^j}^{2l+1,2^m}(f_l)]\leq F$ because of $F_l\leq(\sum_lF_l^2)^{\fr12}$ which is just by an elementary inequality \eqref{e:21mono01}.
Moreover we get
\Bes
\begin{split}
\|F\|_{L_2(\N)}&=(\sum_{l}\|F_l\|_{L_2(\N)}^2)^{\fr12}\lc(\sum_{l}\|f_l\|_{L_2(\N)}^2)^{\fr12}\\
&=\Big(\tau\int_{\R^2}\sum_{l}\chi_{\Ga_{2l+1}\cup\Ga_{2l}}(\xi)|\wh{f}(\xi)|^2d\xi\Big)^{1/2}\lc\|f\|_{L_2(\N)}
\end{split}
\Ees
where in the third equality we use vector-valued Plancherel's theorem \eqref{e:21planch}, the last inequality just follows from $\Ga_{1},\cdots,\Ga_{2^{m-1}-1}$ are disjoint from each other and vector-valued Plancherel's theorem \eqref{e:21planch} again. Thus we prove that
\Bes
\big\|\sup_{0\leq l<2^{m-1}\atop j\in\Z}\Re[M_{2^j}^{2l+1,2^m}(f_l)]\big\|_{L_2(\N)}\lc\|f\|_{L_2(\N)}.
\Ees
By applying the similar argument to the imaginary part, we could also get
\Bes
\big\|\sup_{0\leq l<2^{m-1}\atop j\in\Z}\Im[M_{2^j}^{2l+1,2^m}(f_l)]\big\|_{L_2(\N)}\lc\|f\|_{L_2(\N)}.
\Ees
Combining these estimates of real and imaginary parts, we obtain the desired estimate of $II_1$.

For the term $II_2$, using the similar argument as we have done in the proof of $II_1$, we also get that $II_2$ is bounded by $\|f\|_{L_2(\N)}$.

At last we turn to the term $II_3$. We first introduce an inequality as follows:
\Be\label{e:21Cach}
|\vphi(ab)|^2\leq\vphi(|a|b)\vphi(|a^*|b),\quad \forall\ a,b\in\N\  \text{with}\ b\geq0.
\Ee
This inequality could be verified by writing $a$ as the polar decomposition $a=u|a|$ and using Cauchy-Schwarz's inequality,
\Bes
\begin{split}
|\vphi(ab)|^2&=|\vphi(b^{1/2}u|a|b^{1/2})|^2\leq\vphi(b^{1/2}u|a|u^*b^{1/2})\vphi(b^{1/2}|a|b^{1/2})\\
&=\vphi(b^{1/2}|a^*|b^{1/2})\vphi(b^{1/2}|a|b^{1/2})=\vphi(|a^*|b)\vphi(|a|b).
\end{split}
\Ees

For simplicity, we define $B_{2^j}^{l,2^m}(g)=M_{2^j}^{2l+1,2^m}(g)-M_{2^j}^{2l,2^m}(g)$. Then by the above inequality \eqref{e:21Cach} and Cauchy-Schwarz's inequality, we have
\Bes
\begin{split}
II_3&\leq\sum_{0\leq l<2^{m-1}\atop j\in\Z}\vphi(|B_{2^j}^{l,2^m}(r_l)|h_{2^j}^{2l+1})^{\fr12}\vphi(|B_{2^j}^{l,2^m}(r_l)^*|h_{2^j}^{2l+1})^{\fr12}\\
&\leq\vphi\Big(\sum_{0\leq l<2^{m-1}\atop j\in\Z}|B_{2^j}^{l,2^m}(r_l)|h_{2^j}^{2l+1}
\Big)^{\fr12}\cdot\vphi\Big(\sum_{0\leq l<2^{m-1}\atop j\in\Z}|B_{2^j}^{l,2^m}(r_l)^*|h_{2^j}^{2l+1}
\Big)^{\fr12}\\
&\leq\Big\|\sup_{0\leq l<2^{m-1}\atop j\in\Z}|B_{2^j}^{l,2^m}(r_l)|\Big\|_{L_2(\N)}^{\fr12}
\cdot\Big\|\sup_{0\leq l<2^{m-1}\atop j\in\Z}|B_{2^j}^{l,2^m}(r_l)^*|\Big\|_{L_2(\N)}^{\fr12},
\end{split}
\Ees
where in the last inequality we apply the dual property (ii) in Lemma \ref{l:21dual}.

We first consider the part $|B_{2^j}^{l,2^m}(r_l)|$. Our goal is to show that
\Be\label{e:21Bl}
\Big\|\sup_{0\leq l<2^{m-1}\atop j\in\Z}|B_{2^j}^{l,2^m}(r_l)|\Big\|_{L_2(\N)}\lc\|f\|_{L_2(\N)}.
\Ee
The strategy here is to use a square function to control the maximal function, which has been appeared in the proof of $II_1$. In fact applying an equivalent norm of $L_p(\M;\ell_\infty)$ in \eqref{e:21positivem},   monotone properties in \eqref{e:21mono01} and \eqref{e:21mono}, we have
\Bes
\Big\|\sup_{0\leq l<2^{m-1}\atop j\in\Z}|B_{2^j}^{l,2^m}(r_l)|\Big\|_{L_2(\N)}
\leq\Big\|\Big(\sum_{0\leq l<2^{m-1}\atop j\in\Z}|B_{2^j}^{l,2^m}(r_l)|^2\Big)^{\fr12}\Big\|_{L_2(\N)}.
\Ees
It is straightforward to show that $\mathcal{F}[M_{2^j}^{2l,2^m}(g)](\xi)=\wh\psi(2^{j+m}\inn{e^{2l}_m}{\xi})\widehat{g}(\xi)$.
Then utilizing vector-valued Plancherel's theorem \eqref{e:21planch} and the definitions of $B_{2^j}^{l,2^m}$, $r_l$, we get that
\Bes
\begin{split}
&\Big\|\Big(\sum_{0\leq l<2^{m-1}\atop j\in\Z}|B_{2^j}^{l,2^m}(r_l)|^2\Big)^{\fr12}\Big\|_{L_2(\N)}^2
\\
&=\tau\int_{\R^2}\sum_{j\in\Z}\sum_{0\leq l<2^{m-1}}\Big|\wh{\psi}\Big(2^{j+m}\inn{e_{m}^{2l+1}}{\xi}\Big)-\wh{\psi}\Big(2^{j+m}\inn{e_{m}^{2l}}{\xi}\Big)\Big|^2
\chi_{\Ga_{2l+1}^c\cap\Ga_{2l}^c}(\xi)|\wh{f}(\xi)|^2d\xi.
\end{split}
\Ees

To prove \eqref{e:21Bl}, again by vector-valued Plancherel's theorem \eqref{e:21planch}, we only need to show the multiplier above is bounded, i.e.
\Be\label{e:21multi}
\sum_{j\in\Z}\sum_{0\leq l<2^{m-1}}\Big|\wh{\psi}\Big(2^{j+m}\inn{e_m^{2l+1}}{\xi}\Big)-\wh{\psi}\Big(2^{j+m}\inn{e_m^{2l}}{\xi}\Big)\Big|^2\chi_{\Ga_{2l+1}^c\cap\Ga_{2l}^c}(\xi)\leq C
\Ee
holds uniformly for $\xi\neq0$.
Fix $\xi\neq0$. It suffices to consider the sum of $l$ such that $\xi\in\Ga_{2l+1}^c\cap \Ga_{2l}^{c}$ which means that
\Bes
|\inn{e_m^{2l+1}}{\xi'}|>c2^{-m},\quad |\inn{e_m^{2l}}{\xi'}|>c2^{-m}
\Ees
where $\xi'=\xi/|\xi|$. Such lower estimates may be not enough to prove \eqref{e:21multi}.
In the following, we obtain some better lower estimates via some geometric observations of the plane. Denote $L=\{l:\ 0\leq l<2^{m-1}, \xi\in\Ga_{2l+1}^c\cap\Ga_{2l}^c\}$.
Then by using the mean value formula and Cauchy-Schwarz's inequality, we see that
\Bes
\begin{split}
\text{LHS}\eqref{e:21multi}&
\leq\sum_{j\in\Z}\sum_{l\in L}\Big[\int_0^1 2^{j}|\xi|\cdot|\nabla \wh{\psi}(2^{j+m}\inn{e_{m}^{2l+s}}{\xi})|ds\Big]^2\\
&\leq\int_0^1\sum_{j\in\Z}\sum_{l\in L}\big[ 2^{j}|\xi|\cdot|\nabla \wh{\psi}(2^{j+m}\inn{e_m^{2l+s}}{\xi})|\big]^2ds.
\end{split}
\Ees

Let $e^\bot$ denote the orthogonal unit vector of $e$ in $\R^2$.
Then we observe that  $e_{m}^{l,\bot}$s, where $l=0,\cdots,2^{m-1}-1$, are equally distributed with distance $2^{-m}$ in the plane $\{(x,y): |x|\leq|y|, x\leq0,y\geq0\}$ or $\{(x,y): |x|\leq|y|, x\geq0,y\leq0\}$(see Figure \ref{f:21max}  below).

\begin{figure}[!h]
\centering
\includegraphics[height=0.6\textwidth]{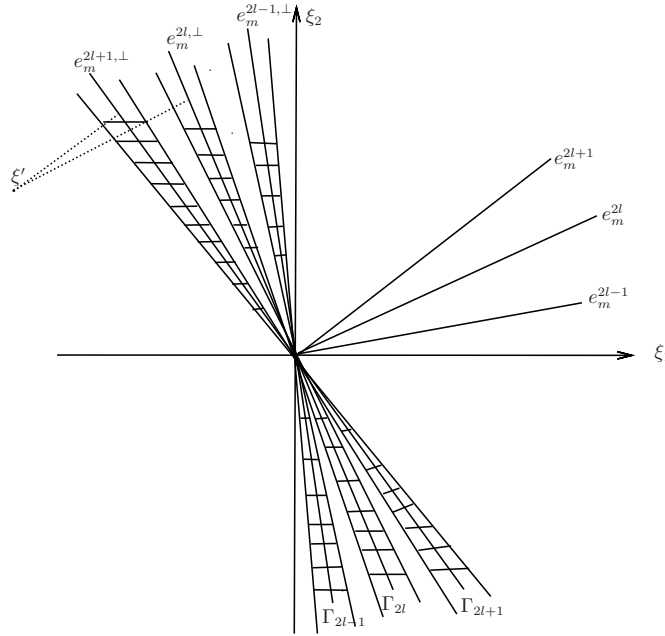}
\caption{\small $\{e_m^l\}_l$ is equally distributed. $\Ga_l$s are disjoint from each other. $|\inn{e_m^{2l}}{\xi'}|$ is the distance between the point $\xi'$ and the line along the direction $e^{2l,\bot}_m$.}\label{f:21max}
\end{figure}

Notice that the geometric illustration of $|\inn{e}{\xi'}|$ is the distance between the point $\xi'$ and the line along the direction $e^\bot$.
Recall that $\xi'$ is a fixed unit vector. By the equally distributed property of $\{e_{m}^{l,\bot}\}_{l\in\{0,\cdots,2^{m-1}-1\}}$, we can separate the set $L$ into at most $2^{m-1}$ subsets $L_1,\cdots,L_{2^{m-1}}$ such that
each $L_i$ has cardinality less than some absolute constant and for each $i$,
\Bes
|\inn{e_m^{2l+1}}{\xi'}|>ci2^{-m},\quad |\inn{e_m^{2l}}{\xi'}|>ci2^{-m},\quad\forall l\in L_i.
\Ees
This fact can be easily seen from the geometry in Figure \ref{f:21max}.
Therefore by the above inequality, for any $s\in [0,1]$, $l\in L_i$, we get $|\inn{e_m^{2l+s}}{\xi'}|>ci2^{-m}$. Since $\wh{\psi}$ is a Schwartz function, so $\nabla \wh{\psi}$ is smooth and rapidly decays at infinity. We finally conclude that
\Bes
\begin{split}
\text{LHS}\eqref{e:21multi}
&\leq\sum_{j\in\Z}\sum_{1\leq i\leq2^{m-1}} 2^{2j}|\xi|^2\min\{1,(2^{j}|\xi|i)^{-100}\}\\
&\lc\sum_{1\leq i\leq 2^{m-1}}i^{-2}\sum_{j\in\Z}[i2^{j}|\xi|]^2[1+(2^{j}|\xi|i)]^{-100}\leq C
\end{split}
\Ees
which proves \eqref{e:21Bl}. For the term $|B_{2^j}^{l,2^m}(r_l)^*|$, we can also obtain
\Bes
\Big\|\sup_{0\leq l<2^{m-1}\atop j\in\Z}|B_{2^j}^{l,2^m}(r_l)^*|\Big\|_{L_2(\N)}\lc\|f\|_{L_2(\N)}
\Ees
by using the same argument if observing that
\Bes
\F[B_{2^j}^{l,2^m}(r_l)^*](\xi)=\Big[\wh{\psi}\Big(-2^{j+m}\inn{e_{m}^{2l+1}}{\xi}\Big)-\wh{\psi}\Big(-2^{j+m}\inn{e_{m}^{2l}}{\xi}\Big)\Big]
\chi_{\Ga_{2l+1}^c\cap\Ga_{2l}^c}(-\xi)\wh{f}(\xi).
\Ees
Combining these estimates of $|B_{2^j}^{l,2^m}(r_l)|$ and $|B_{2^j}^{l,2^m}(r_l)^*|$, we get $II_3$ is majorized by $\|f\|_{L_2(\N)}$ which ends the proof this theorem.
\end{proof}
\vskip0.24cm

\section{Bochner-Riesz means on $L_\infty(\R^2)\tens\M$}\label{s:21BR}

In this section, we study the operator-valued Bochner-Riesz means on $L_p(\N)$, where $\N=L_\infty(\R^2)\tens\M$ throughout this section.
The main tools  are noncommutative Kakeya maximal functions which have been investigated in the previous section. Our main result can be stated as follows.

\begin{theorem}\label{t:21BR}
Let $0<\lam\leq\fr12$ and $\fr{4}{3+2\lam}<p<\fr{4}{1-2\lam}$. Then for the Bochner-Riesz means  $B^\lam_R$ given in \eqref{e:21BRR}, set $B^\lam=B_1^\lam$, we have
\Be\label{e:21BR}
\|B^\lam(f)\|_{L_p(\N)}\lc\|f\|_{L_p(\N)},\quad
\sup_{R>0}\|B_R^\lam(f)\|_{L_p(\N)}\lc\|f\|_{L_p(\N)}.
\Ee
Consequently for $f\in L_p(\N)$, $B_R^\lam(f)$ converges to $f$ in $L_p(\N)$ as $R\rta\infty$.
\end{theorem}

Before giving the proof, let us start with some definitions and lemmas.
We first define the following Fourier multiplier on $L_p(\N)$ for convenience.

\begin{definition}
We say $m:\R^2\rta\C$ is an $L_p(\N)$ Fourier multiplier if the operator $T_m$ defined by
\Bes
\wh{T_m(f)}(\xi)=m(\xi)\wh{f}(\xi),\quad f\in\mathbf{S}(\R^2)\otimes \S,
\Ees
extends to a bounded operator on $L_p(\N)$, where $\mathbf{S}(\R^2)$ is the class of Schwartz functions on $\R^2$ and $\S$ is the linear span of all $x\in\M_{+}$ whose support projections have finite trace defined in Subsection \ref{s:212lp}. Denote by $M_p(\N)$ the space of all $L_p(\N)$ multipliers and $\|\cdot\|_{M_p(\N)}$ the $L_p$ multiplier norm.
\end{definition}
Notice that for $\Re \lam>\fr{1}{2}$, the convolution kernel associated to $(1-|\cdot|^2)^\lam_+$ is integrable over $\R^2$ with the bound $e^{6|\Im\lam|^2}$ (see e.g. \cite[Proposition 5.2.2]{Gra14M}). Then we immediately get the following lemma.
\begin{lemma}\label{l:21bcritical}
Let $\lam\in\C$. If $\Re \lam>\fr{1}{2}$, then for all $1\leq p\leq\infty$, we have
\Bes
\|(1-|\cdot|^2)_+^\lam\|_{M_p(\N)}\lc e^{6|\Im\lam|^2}.
\Ees
\end{lemma}

Next we introduce a noncommutative Littlewood-Paley-Rubio de Francia's square function inequality for equal intervals whose proof can be found in \cite{HTY09}. This inequality is also a key step in our study of Bochner-Riesz means.
\begin{lemma}\label{t:21Little}
Set $\N=L_\infty(\R)\tens\M$. Let $I_j$s be intervals of equal length with disjoint interior, $j\in\Z$ and $\bigcup_{j\in\Z}I_j=\R$. Define $\wh{P_jf}(\xi)=\chi_{I_j}(\xi)\wh{f}(\xi)$. Then for all $2\leq p<\infty$, we have
\Bes
\|\{P_j(f)\}\|_{L_p(\N;\ell_2^{rc})}\lc\|f\|_{L_p(\N)}.
\Ees
\end{lemma}

Now we begin to prove Theorem \ref{t:21BR}.
\begin{proof}[Proof of Theorem \ref{t:21BR}]
We first point out the norm convergence in $L_p(\N)$ is a direct consequence of \eqref{e:21BR}. In fact for any $f\in \mathbf{S}(\R^2)\otimes\S$, i.e. $f(x)=\sum_{i=1}^n\psi_i(x)a_i$ with $\psi_i\in \mathbf{S}(\R^2)$ and $a_i\in\S$, $B_R^\lam(f)$ converges to $f$ in $L_p(\N)$ as $R\rta\infty$.
Then by the density argument and the second inequality in \eqref{e:21BR}, we could get for every $f\in L_p(\N)$, $B_R^\lam(f)$ converges to $f$ in $L_p(\N)$.
So in our proof below, it is sufficient to consider \eqref{e:21BR}.
Since the multiplier in $M_p(\N)$ is invariant under the dilation: $$\|m(R^{-1}\cdot)\|_{M_p(\N)}=\|m\|_{M_p(\N)},$$
we only need to consider the first inequality in \eqref{e:21BR}.
Throughout the proof, we suppose that $\lam$ is a complex number with $\Re\lam>0$. When $p=2$, the estimate just follows from vector-valued Plancherel's theorem \eqref{e:21planch}. By the duality, it is enough to show \eqref{e:21BR} holds for $2<p<\fr{4}{1-2\Re\lam}$.

We first make a dyadic decomposition of the multiplier. To do that, we choose a smooth function $\phi$ supported in $[-\fr12,\fr12]$ and a smooth function $\psi$ supported in $[\fr18,\fr58]$ such that
\Bes
\phi(t)+\sum_{k=0}^\infty\psi(2^k(1-t))=1,\quad\forall t\in[0,1).
\Ees
Using this equality, we can decompose the multiplier $(1-|\xi|^2)_+^\lam$ as follows
\Be\label{e:21BRmud}
\begin{split}
(1-|\xi|^2)_+^\lam&=\Big[\phi(|\xi|)+\sum_{k=0}^\infty\psi(2^k(1-|\xi|))\Big](1-|\xi|^2)^\lam_+\\
&=\phi(|\xi|)(1-|\xi|^2)^\lam+\sum_{k=0}^{\infty}2^{-k\lam}(2^k(1-|\xi|))^\lam\psi(2^k(1-|\xi|))(1+|\xi|)^\lam\\
&=:m_{00}(\xi)+\sum_{k=0}^\infty2^{-k\lam}m_k(\xi).
\end{split}
\Ee

Below we give some observations of Fourier multipliers $m_{0,0}$ and $m_k$.
It is easy to see that $m_{00}$ is a smooth function with compact support, hence the multiplier $m_{00}$ lies in $M_p(\N)$ for all $1\leq p\leq\infty$. Notice that each function $m_k$ is also smooth, radial and supported in a small annulus:
$$\fr{1}{8}2^{-k}\leq1-|\xi|\leq\fr582^{-k}.$$
Therefore each $m_k$ also lies in $M_p(\N)$ for $1\leq p\leq\infty$, but its bound may depend on $k$.
To sum over the series in \eqref{e:21BRmud}, it is crucial to determine  exactly the bound of multiplier $m_k$ in $k$.
Our main goal in this proof is to show that for each $k$,
\Be\label{e:21BRkm}
\|m_k\|_{M_4(\N)}\lc(1+|k|)^{1/2}(1+|\lam|)^3.
\Ee

Suppose we have \eqref{e:21BRkm} for the moment.
Then summing over the series in \eqref{e:21BRmud} with these estimates in \eqref{e:21BRkm}, we get
\Bes
\begin{split}
\|B^\lam(f)\|_{L_4(\N)}&\lc\|\F^{-1}[{m_{00}}]*f\|_{L_4(\N)}+\sum_{k=0}^\infty2^{-k\Re\lam}\|\F^{-1}[m_k]*f\|_{L_4(\N)}\\
&\lc\|f\|_{L_4(\N)}+(1+|\lam|)^3\|f\|_{L_4({\N})}\lc(1+|\lam|)^3\|f\|_{L_4(\N)}.
\end{split}
\Ees
Applying the noncommutative Riesz-Thorin interpolation theorem (see \cite{Xu07}), we get \eqref{e:21BR} for $2<p<4$. Next we use the analytic interpolation theorem (see Theorem \ref{t:21inter} in the appendix) to show the remaining part. Let $\eps>0$ be a small constant (for example less than $\fr12$). Then
\begin{align*}
\|B^\lam(f)\|_{L_4(\N)}&\lc(1+|\Im\lam|)^3\|f\|_{L_4(\N)},& \Re\lam&=\eps;\\
\|B^\lam(f)\|_{L_\infty(\N)}&\lc e^{6|\Im\lam|^2}\|f\|_{L_\infty(\N)}, &\Re\lam&=\fr12+\eps,
\end{align*}
where the second inequality just follows from Lemma \ref{l:21bcritical}.   Define a new operator $T_z$ by
$$T_z(f)=B^{\fr12z+\eps}(f).$$
To apply Theorem \ref{t:21inter}, we should verify the hypothesis of Theorem \ref{t:21inter}.
It is easy to see that $z\rta T_z$ is an analytic family of linear operators with admissible growth and
\Bes
\begin{split}
\|T_{iy}(f)\|_{L_4(\N)}&\leq C_1(1+|y|)^3\|f\|_{L_4(\N)},\\
\|T_{1+iy}(f)\|_{L_\infty(\N)}&\leq C_2e^{\fr32|y|^2}\|f\|_{L_\infty(\N)}.
\end{split}
\Ees
Let $p$ and $\tet$ satisfy $\fr{1}{p}=\fr{1-\tet}{4}+\fr{\tet}{\infty}$ and $\fr{1}{2}\tet+\eps=\Re\lam$. Set
$M_0(y)=C_1(1+|y|)^3$ and  $M_1(y)=C_2e^{\fr32|y|^2}$. Then applying Theorem \ref{t:21inter}, we get
\Bes
\|T_{\tet}(f)\|_{L_p(\N)}\leq M(\tet)\|f\|_{L_p(\N)},
\Ees
where $M(\tet)$ is a finite constant defined in Theorem \ref{t:21inter}. This immediately implies our required estimate \eqref{e:21BR} since $B^{\Re\lam}(f)=T_\tet(f)$ and $p=\fr{4}{1-2\Re\lam+2\eps}<\fr{4}{1-2\Re\lam}$.

Now we turn to prove \eqref{e:21BRkm}. Fix $k$ in the rest of proof.
To estimate the $L_4$ norm of the Fourier multiplier $m_k$, we need an additional decomposition (the called microlocal decomposition) of  $m_k$ that takes into the radial nature of $m_k$.
We usually identify the plane $\R^2$ as complex plane $\C$ by $(x,y)\leftrightarrow z=x+iy$.
The reason to do this is that the expression $z=re^{2\pi i\tet}$ in $\C$ can be easily understood in the geometric point (note that $r$ is length of $|z|$ and $\tet$ is an argument).
Next we define  sectorial arcs as follows:
\Bes
\Ga_{k,l}=\{re^{2\pi i\tet}:\ (r,\tet)\in\R_+\times[0,1), |\tet-l2^{-\fr k2}|<2^{-\fr k2}, 1-\fr58 2^{-k}\leq r\leq 1-\fr 18 2^{-k}\},
\Ees
for all $l\in\{0,1,2,\cdots,\lfloor2^{k/2}\rfloor\}$. We choose a smooth function $\om$ such that $\om(u)=1$ for $|u|<\fr14$, $\om(u)=0$ for $|u|>1$ and $\sum_{l\in\Z}\om(x-l)=1$ holds for all $x\in\R$.
 Define
$$m_{k,l}(re^{2\pi i\tet})=m_k(re^{2\pi i\tet})\om(2^{k/2}\tet-l).$$
It should be pointed out that $m_k(re^{2\pi i\tet})$ is a function about variable $r$.
Then by our construction of $\om$, we see that
\Bes
m_k(\xi)=\sum_{l\in\Z}m_{k,l}(\xi)=\sum_{l=0}^{r(k)}m_{k,l}(\xi),
\Ees
where $r(k)=\lfloor2^{k/2}\rfloor-1$ if $k$ is even and $r(k)=\lfloor 2^{k/2}\rfloor$ if $k$ is odd. By some elementary calculations, it is not difficult to get that
\Be\label{e:21parirtet}
|\pari_r^\alp\pari_\tet^\beta m_{k,l}(re^{2\pi i\tet})|\lc(1+|\lam|)^{\alp}2^{k\alp}2^{\fr k2 \beta},\quad \supp\ m_{k,l}(re^{2\pi i\tet})\subset \Ga_{k,l}.
\Ee

Next we  split all $\{m_{k,l}\}_l$ into five subsets whose supports satisfy the following conditions:

\begin{enumerate}[\quad\quad (a).]
\item  $\supp m_{k,l}\subsetneq Q_a=:\{(x,y)\in\R^2: x>0,|y|<|x|\};$
\item  $\supp m_{k,l}\subsetneq Q_b=:\{(x,y)\in\R^2: x<0,|y|<|x|\};$
\item  $\supp m_{k,l}\subsetneq Q_c=:\{(x,y)\in\R^2: y>0,|y|>|x|\};$
\item  $\supp m_{k,l}\subsetneq Q_d=:\{(x,y)\in\R^2: y<0,|y|>|x|\};$
\item  The support $\supp m_{k,l}$ intersects $Q_e=\{(x,y)\in\R^2:|x|=|y|\}.$
\end{enumerate}

We first observe that there are only at most eight $m_{k,l}$s in the case (e). In such a case, it is straightforward to get that
\Be\label{e:21ecase}
\|m_{k,l}\|_{M_p(\N)}\lc(1+|\lam|)^3
\Ee
as we will see below (which will be pointed out in our later proof). By symmetry, we only need to concentrate on the case (a). Denote the index set $\mathfrak{I}=\{l:\ \Ga_{k,l}\subsetneq Q_a\}$.
For each $l\in\mathfrak{I}$, define a Fourier multiplier operator $T_{k,l}$ by
\Bes
\widehat{T_{k,l}g}(\xi)={m_{k,l}}(\xi)\widehat{g}(\xi).
\Ees
Our purpose below is to show $\sum_{l\in\mathfrak{I}}m_{k,l}$ satisfies \eqref{e:21BRkm}, i.e.
\Be\label{e:21BRk}
\Big\|\sum_{l\in\mathfrak{I}}T_{k,l}(f)\Big\|_{L_4(\N)}\lc(1+|k|)^{\fr12}(1+|\lam|)^3\|f\|_{L_4(\N)}.
\Ee

By decomposing $f$ as linear combination of four positive functions, we may suppose that $f$ is positive.
In the following we should separate the sum of $l\in\mathfrak{I}$ into two parts,
\Be\label{e:21Tkl}
\begin{split}
\Big\|\sum_{l\in\mathfrak{I}}T_{k,l}(f)\Big\|_{L_4(\N)}^4
&=\tau\int_{\R^2}|\sum_{l\in\mathfrak{I}}\sum_{l'\in\mathfrak{I}}T_{k,l}(f)^*(x)T_{k,l'}(f)(x)|^2dx\\
&\lc\tau\int_{\R^2}|\sum_{l\in\mathfrak{I}}\sum_{l'\in\mathfrak{I},|l-l'|\leq 10^3}T_{k,l}(f)^*(x)T_{k,l'}(f)(x)|^2dx\\
&\quad+\tau\int_{\R^2}|\sum_{l\in\mathfrak{I}}\sum_{l'\in\mathfrak{I},|l-l'|>10^3}T_{k,l}(f)^*(x)T_{k,l'}(f)(x)|^2dx\\
&=: I+II.
\end{split}
\Ee

For the term $I$, the sum of $l'$ is taking over $l-10^3\leq l'\leq l+10^3$  which is finite. By  H\"older's inequality of the square function in Lemma \ref{l:21vecth}, we get
\Bes
\begin{split}
I&\lc\sum_{i=-10^3}^{10^3}\tau\int_{\R^2}|\sum_{l\in\mathfrak{I},l+i\in\mathfrak{I}}T_{k,l}(f)^*(x)T_{k,l+i}(f)(x)|^2dx\\
&\leq\sum_{i=-10^3}^{10^3}\Big\|\Big(\sum_{l\in\mathfrak{I}}|T_{k,l}(f)|^2\Big)^{\fr12}\Big\|^2_{L_4(\N)}
\Big\|\Big(\sum_{l+i\in\mathfrak{I}}|T_{k,l+i}(f)|^2\Big)^{\fr12}\Big\|^2_{L_4(\N)}\\
&\lc\Big\|\Big(\sum_{l\in\mathfrak{I}}|T_{k,l}(f)|^2\Big)^{\fr12}\Big\|^4_{L_4(\N)}.
\end{split}\Ees

For the term $II$, using vector-valued Plancherel's theorem \eqref{e:21planch}, we get
$$II=\tau\int_{\R^2}\big|\sum_{l\in\mathfrak{I}}\sum_{l'\in\mathfrak{I},|l-l'|>10^3}\widehat{T_{k,l}(f)^*}*\widehat{T_{k,l'}(f)}(\xi)\big|^2 d\xi.$$
Since $f$ is positive, it is easy to see $T_{k,l}(f)^*=\F^{-1}[{\widetilde{m_{k,l}}}]*f$ where we use the notation $\tilde{a}(\cdot)=a(-\cdot)$. Then we have
\Bes
\supp\widehat{T_{k,l}(f)^*}\subset\tilde{\Ga}_{k,l},\quad
\supp\widehat{T_{k,l'}(f)}\subset\Ga_{k,l'},
\Ees
where $\tilde{\Ga}_{k,l}=\{x: -x\in\Ga_{k,l}\}$.
Applying these support conditions, we rewrite the above integral of $II$ as
$$\tau\int_{\R^2}\big|\sum_{l\in\mathfrak{I}}\sum_{l'\in\mathfrak{I},|l-l'|>10^3}\widehat{T_{k,l}(f)^*}*\widehat{T_{k,l'}(f)}(\xi)\chi_{\tilde\Ga_{k,l}+{\Ga}_{k,l'}}(\xi)\big|^2 d\xi.$$
Next using the convexity operator inequality \eqref{e:21conv}, this term is bounded by
$$\tau\int_{\R^2}\Big(\sum_{l\in\mathfrak{I}}\sum_{l'\in\mathfrak{I},|l-l'|>10^3}|\widehat{T_{k,l}(f)^*}*\widehat{T_{k,l'}(f)}(\xi)|^2\Big)\Big(\sum_{l\in\mathfrak{I}}\sum_{l'\in\mathfrak{I},|l-l'|>10^3}\chi_{\tilde\Ga_{k,l}+{\Ga}_{k,l'}}(\xi)\Big)d\xi.$$

Before proceeding our proof further, we need the following geometric estimate.
\begin{lemma}\label{l:21Geo}
There exists a constant $C$ independent of $k$ and $\xi$ such that
$$\sum_{l\in\mathfrak{I}}\sum_{l'\in\mathfrak{I},|l-l'|>10^3}\chi_{\tilde\Ga_{k,l}+{\Ga}_{k,l'}}(\xi)\leq C.$$
\end{lemma}
\begin{remark}
The geometric estimate of this lemma is different from the classical commutative one stated in \cite{Cor79}. The main difference is that there is no involution $*$ of $T_{k,l}f$ in the commutative case because of $|a|^4=|aa|^2$ if $a$ is a complex number. This yields a geometric estimate in the form $\sum_{l\in\mathfrak{I}}\sum_{l'\in\mathfrak{I}}\chi_{\Ga_{k,l}+\Ga_{k,l'}}(\xi)$ which is simpler than that of Lemma \ref{l:21Geo}.
In fact, all $\Ga_{k,l}$s are contained in $\{(x,y):x>0,|x|<|y|\}$ and the pieces $\Ga_{k,l}+\Ga_{k,l'}$s are well distributed.
However in the noncommutative setting, for an element $a\in\M$, we have $|a|^4=|a^*a|^2$.
Our argument above leads to the geometric estimate $\sum_{l\in\mathfrak{I}}\sum_{l'\in\mathfrak{I}}\chi_{\tilde{\Ga}_{k,l}+\Ga_{k,l'}}(\xi)$ .
Notice $\tilde{\Ga}_{k,l}$ stays in an oppositive direction of $\Ga_{k,l}$, i.e. $\tilde{\Ga}_{k,l}\subset\{(x,y): x<0, |x|<|y|\}$.
These pieces $\tilde{\Ga}_{k,l}+\Ga_{k,l'}$s may accumulate near the origin if $\Ga_{k,l}$ and $\Ga_{k,l'}$ are close enough, which can be easily seen in view of geometric observation.
Hence if $\xi$ is close to the origin, $\sum_{l\in\mathfrak{I}}\sum_{l'\in\mathfrak{I}}\chi_{\tilde{\Ga}_{k,l}+\Ga_{k,l'}}(\xi)$ may be infinite.
This is the reason why we split the sum of $\mathfrak{I}$ into two parts in \eqref{e:21Tkl}.
Luckily we can show  $\{\tilde{\Ga}_{k,l}+\Ga_{k,l'}\}_{l,l'}$ is finite overlapped if $|l-l'|>10^3$.
\end{remark}
The proof of Lemma \ref{l:21Geo} will be given later. Applying this geometric estimate, we get $II$ is bounded by
\Bes
\begin{split}
\sum_{l\in\mathfrak{I}}\sum_{l'\in\mathfrak{I}}\tau\int_{\R^2}|\widehat{T_{k,l}(f)^*}&*\widehat{T_{k,l'}(f)}(\xi)|^2 d\xi=\sum_{l\in\mathfrak{I}}\sum_{l'\in\mathfrak{I}}\tau\int_{\R^2}|{T_{k,l}(f)^*(x)}{T_{k,l'}(f)}(x)|^2 dx\\
&=\sum_{l\in\mathfrak{I}}\sum_{l'\in\mathfrak{I}}\tau\int_{\R^2}{T_{k,l'}(f)^*(x)}{T_{k,l}(f)}(x){T_{k,l}(f)^*(x)}{T_{k,l'}(f)}(x) dx\\
&=\tau\int_{\R^2}\Big(\sum_{l\in\mathfrak{I}}|{T_{k,l}(f)^*(x)}|^2\Big)^2 dx
=\Big\|\Big(\sum_{l\in\mathfrak{I}}|{T_{k,l}(f)^*}|^2\Big)^{\fr12} \Big\|^4_{L_4(\N)},
\end{split}
\Ees
where in the first equality we use vector-value's Plancherel theorem \eqref{e:21planch} and the third equality follows from the tracial property of $\tau$: $\tau(ab)=\tau(ba)$.
Combining these estimates of $I$ and $II$, together with \eqref{e:21Tkl},
we get
\Be\label{e:21Tklsqu}
\Big\|\sum_{l\in\mathfrak{I}}T_{k,l}(f)\Big\|_{L_4(\N)}^4\lc\Big\|\Big(\sum_{l\in\mathfrak{I}}|{T_{k,l}(f)}|^2\Big)^{\fr12} \Big\|^4_{L_4(\N)}+
\Big\|\Big(\sum_{l\in\mathfrak{I}}|{T_{k,l}(f)^*}|^2\Big)^{\fr12} \Big\|^4_{L_4(\N)}.
\Ee
This kind of estimate is consistent with the estimate of square functions  in noncommutative analysis in the sense that
we should use both row and column square functions to control the $L_p$ norm for $p\geq 2$ (see Subsection \ref{s:212squa}).

Now we should study carefully the multiplier $m_{k,l}$. Consider $m_{k,0}$ firstly.
Note that $m_{k,0}$ is supported in a rectangle parallel to axes with side length $2^{-k-1}$ and $2^{-\fr k2 +1}$ (along the $\xi_1$-axis and $\xi_2$-axis, respectively).
Roughly speaking in such a rectangle $\pari_{\xi_1}\approx\pari_{r}, \pari_{\xi_2}\approx \pari_\tet$. In fact by some straightforward calculations, we get the smooth function $m_{k,0}$ satisfies
\Be\label{e:21mk0}
|\pari_{\xi_1}^\alp\pari_{\xi_2}^\beta m_{k,0}(\xi_1,\xi_2)|\lc_{\alp,\beta}(1+|\lam|)^{\alp}2^{k\alp}2^{\fr k2\beta}
\Ee
for all positive integers $\alp$ and $\beta$ which is an analogue of \eqref{e:21parirtet}. We can rewrite \eqref{e:21mk0} as follows
\Bes
|\pari_{\xi_1}^\alp\pari_{\xi_2}^\beta [m_{k,0}(2^{-k}\xi_1,2^{-\fr k2}\xi_2)]|\lc_{\alp,\beta}(1+|\lam|)^{\alp}.
\Ees
With these estimates and using integration by parts, we obtain that
\Be\label{e:21mk0l}
2^{\fr32k}|\F^{-1}[m_{k,0}](2^kx_1,2^{\fr k2}x_2)|\lc(1+|\lam|)^3(1+|x_1|+|x_2|)^{-3}.
\Ee
Let $v_l$ and $v_l^{\bot}$ be the unit vectors corresponding the directions $e^{2\pi i l2^{-\fr k2}}$ and $ie^{2\pi il 2^{-\fr k2}}$, respectively.
According our definition, $m_{k,l}(\xi)=m_{k,0}(A\xi)$ where $A=\Big({v_l\atop v_l^{\bot}}\Big)$.
By a rotation, we see that
\Be\label{e:21rotation}
\F^{-1}[m_{k,l}](x_1,x_2)=\F^{-1}[m_{k,0}](Ax)=\F^{-1}[m_{k,0}](x\cdot v_l,x\cdot v_l^{\bot}).
\Ee
Hence combining with \eqref{e:21mk0l}, we have
\Bes
|\F^{-1}[m_{k,l}](x_1,x_2)|\lc 2^{-\fr32 k}(1+|\lam|)^3(1+2^{-k}|x\cdot v_l|+2^{-\fr k2}|x\cdot v_l^\bot|)^{-3},
\Ees
which immediately implies that
\Be\label{e:21mkl}
\sup_{k>0}\sup_{l\in\mathfrak{I}}\|\F^{-1}[m_{k,l}]\|_{L_1(\R^2)}\lc(1+|\lam|)^3.
\Ee
Here we point out that \eqref{e:21ecase} follows  from \eqref{e:21mkl}. Obviously all the above estimates hold for $m_{k,l}^*$ since $m_{k,l}^*=\overline{m_{k,l}}$.

We turn to give some geometric observations of the support of $m_{k,l}$.
Let $J_{k,l}$ be the $\xi_2$ projection of the support of $m_{k,l}$. If the support of $m_{k,l}$ lies in the upper half plane of $Q_a$ (i.e. $\xi_2>0$), then we see
\Bes
J_{k,l}=\R\times[(1-\tfrac58 2^{-k})\sin(2\pi 2^{-\fr k2}(l-1)), (1-\tfrac18 2^{-k})\sin(2\pi 2^{-\fr k2}(l+1))].
\Ees
Similarly if the support of $m_{k,l}$ lies in the lower half plane of $Q_a$ (i.e. $\xi_2\leq 0$), then
\Bes
J_{k,l}=\R\times[(1-\tfrac18 2^{-k})\sin(2\pi 2^{-\fr k2}(l-1)), (1-\tfrac58 2^{-k})\sin(2\pi 2^{-\fr k2}(l+1))].
\Ees
Since $k$ is a fixed integer, the sets $J_{k,l}$s are almost disjoint for different $l\in\mathfrak{I}$ because $\Ga_{k,l}$ is only joint with $\Ga_{k,l-1}$ and $\Ga_{k,l+1}$.
Next our goal is to construct congruent strips containing $J_{k,l}$. If we set $\tilde{J}_{k,l}$ as a strip centered at $\xi_2=(1-\fr 382^{-k})\sin(2\pi 2^{-\fr k2}l)$ with width $20\cdot 2^{-\fr{k}{2}}$, i.e.
$$
\tilde{J}_{k,l}=\R\times[(1-\tfrac38 2^{-k})\sin(2\pi 2^{-\fr k2}l)-10\cdot 2^{-\fr k2},(1-\tfrac 38 2^{-k})\sin(2\pi 2^{-\fr k2}l)+10\cdot2^{-\fr k2}].
$$
Then $J_{k,l}\subset \tilde{J}_{k,l}$. For any $\si\in\Z$, $\upsilon\in\{0,1,\cdots,39\}$,  define the strips as follows
\Bes
S_{k,\si,\ups}=\R\times[40\si2^{-\fr k2}+\ups2^{-\fr k2},40(\si+1)2^{-\fr k2}+\ups 2^{-\fr k2}].
\Ees
Note that $S_{k,\si,\ups}$ has width $40\cdot 2^{-\fr k2}$ and each $\tilde{J}_{k,l}$ must be contained in one of $S_{k,\si,\ups}$ for some $\si\in\Z$ and $\ups\in\{0,1,\cdots,39\}$ in view of some simple geometric observation of $S_{k,\si,\ups}$.
We say $\tilde{J}_{k,l}\subset S_{k,\si_l,\ups_l}$ and define $S_{k,\si_l,\ups_l}=B_{k,l}$.
Let $f_{k,l}=\F^{-1}[\chi_{B_{k,l}}\widehat{f}]=(\F^{-1}[\chi_{B_{k,l}}])*f$. According the definition of $T_{k,l}$, we see
\Bes
T_{k,l}(f)=\F^{-1}[m_{k,l}]*f_{k,l}=T_{k,l}(f_{k,l}).
\Ees

Now we come back to the estimate in \eqref{e:21Tklsqu}. We first consider the column square function.
Using the convexity operator inequality \eqref{e:21conv} and  the uniform boundedness of $\F^{-1}[m_{k,l}]$ in \eqref{e:21mkl}, we get
\Bes
\begin{split}
|T_{k,l}(f)(x)|^2&=\Big|\int_{\R^2}\F^{-1}[m_{k,l}](x-y)f_{k,l}(y)dy\Big|^2\\
&\leq\int_{\R^2}|\F^{-1}[m_{k,l}](x-y)|dy\cdot\int_{\R^2}|\F^{-1}
[m_{k,l}](x-y)|\cdot|f_{k,l}(y)|^2dy\\
&\lc(1+|\lam|)^3|\F^{-1}[m_{k,l}]|*|f_{k,l}|^2(x).
\end{split}
\Ees

Plugging the above inequality into the column square function in \eqref{e:21Tklsqu} and applying the monotone property \eqref{e:21mono}, we get
\Bes
\begin{split}
\Big\|\Big(\sum_{l\in\mathfrak{I}}|{T_{k,l}(f)}|^2&\Big)^{\fr12} \Big\|^4_{L_4(\N)}
\lc(1+|\lam|)^6\tau\int_{\R^2}\Big(\sum_{l\in\mathfrak{I}}|\F^{-1}[m_{k,l}]|*|f_{k,l}|^2(x)\Big)^2dx\\
&\ \ \ \  =(1+|\lam|)^6\sup_{g}\Big[\tau\int_{\R^2}\sum_{l\in\mathfrak{I}}|\F^{-1}[m_{k,l}]|*|f_{k,l}|^2(x)g(x)dx\Big]^2
\end{split}
\Ees
where in the last equality we use the duality $[L_p(\N)]^*=L_{p'}(\N)$ and the supremum is taken over all positive $g$ in $L_2(\N)$ with norm less than one. Continue the above estimate, we obtain
\Bes
\begin{split}
&\quad(1+|\lam|)^6\sup_{g}\Big(\tau\sum_{l\in\mathfrak{I}}\int_{\R^2}|f_{k,l}(y)|^2\cdot|\F[m_{k,l}]|*g(y)dy\Big)^2\\
&\leq(1+|\lam|)^6\sup_{g}\|\sup_{l\in\mathfrak{I}}|\F[m_{k,l}]|*g\|_{L_2(\N)}^2\big\|\sum_{l\in\mathfrak{I}}|f_{k,l}|^2\big\|_{L_2(\N)}^2\\
&=(1+|\lam|)^6\sup_{g}\|\sup_{l\in\mathfrak{I}}|\F[m_{k,l}]|*g\|_{L_2(\N)}^2\big\|(\sum_{l\in\mathfrak{I}}|f_{k,l}|^2)^{\fr12}\big\|_{L_4(\N)}^4,
\end{split}
\Ees
where the first inequality follows from the duality in Lemma \ref{l:21dual}. If we can prove the following two estimates:
\Be\label{e:21l2max}
\|\sup_{l\in\mathfrak{I}}|\F[m_{k,l}]|*g\|_{L_2(\N)}\lc(1+|\lam|)^3(1+k)\|g\|_{L_2(\N)},
\Ee
\Be\label{e:21l4}
\big\|(\sum_{l\in\mathfrak{I}}|f_{k,l}|^2)^{\fr12}\big\|_{L_4(\N)}\lc\|f\|_{L_4(\N)},
\Ee
then we finally get that
\Bes
\Big\|\Big(\sum_{l\in\mathfrak{I}}|{T_{k,l}(f)}|^2\Big)^{\fr12} \Big\|_{L_4(\N)}\lc(1+|\lam|^3)(1+k)^{\fr12}\|f\|_{L_4(\N)},
\Ees
which is the required estimate of \eqref{e:21BRk}.
For the row square function, we can use the similar method to obtain the desired estimate
\Bes
\Big\|\Big(\sum_{l\in\mathfrak{I}}|{T_{k,l}(f)}^*|^2\Big)^{\fr12} \Big\|_{L_4(\N)}\lc(1+|\lam|^3)(1+k)^{\fr12}\|f\|_{L_4(\N)}.
\Ees
In fact, $T_{k,l}(f)^*=\F[m_{k,l}^*]*f_{k,l}^*$. Then repeating these arguments for the column square function above, we see the proof is reduced to the following two inequalities:
\Be\label{e:21l2maxself}
\|\sup_{l\in\mathfrak{I}}|\F[m_{k,l}^*]|*g\|_{L_2(\N)}\lc(1+|\lam|)^3(1+k)\|g\|_{L_2(\N)},
\Ee
\Be\label{e:21l4self}
\big\|(\sum_{l\in\mathfrak{I}}|f_{k,l}^*|^2)^{\fr12}\big\|_{L_4(\N)}\lc\|f\|_{L_4(\N)}.
\Ee

Below we give the proofs of \eqref{e:21l2max}, \eqref{e:21l2maxself} and \eqref{e:21l4}, \eqref{e:21l4self}.
We first consider \eqref{e:21l2max} and \eqref{e:21l2maxself}.
Recall \eqref{e:21mk0l}, $\F[m_{k,0}]$ is integrable over $\R^2$ and satisfies $$2^{\fr32 k}|\F[m_{k,0}](2^{k}x_1,2^{\fr k2}{x_2})|\lc\fr{(1+|\lam|)^3}{(1+|x|)^3}\lc(1+|\lam|)^3\sum_{s=0}^\infty \fr{2^{-s}}{2^{2s}}\chi_{[-2^s,2^s]\times[-2^s,2^s]}(x).$$

By making a change of variables, we immediately get
\Bes
|\F[m_{k,0}](x)|\lc(1+|\lam|)^3\sum_{s=0}^\infty{2^{-s}}\fr{1}{|R_s|}\chi_{R_s}(x),
\Ees
where $R_s=[-2^{s+k},2^{s+k}]\times [-2^{s+\fr k2},2^{s+\fr k2}]$. Notice that a general function $\F[m_{k,l}]$ is obtained from $\F[m_{k,0}]$ by a rotation $A=\Big({v_l\atop v_l^{\bot}}\Big)$ (see \eqref{e:21rotation}). Therefore we get
\Bes
|\F[m_{k,l}](x)|\lc(1+|\lam|)^3\sum_{s=0}^\infty{2^{-s}}\fr{1}{|R_{s,l}|}\chi_{R_{s,l}}(x),
\Ees
where $R_{s,l}$ is a rectangle with principal axes along the directions $v_l$ and $v_l^{\bot}$ with half side length $2^{s+k}$ and $2^{s+\fr k2}$, respectively. Since $g$ is positive, we see that
\Bes
|\F[m_{k,l}]|*g(x)\lc(1+|\lam|)^3\sum_{s=0}^\infty
2^{-s}\fr{1}{|R_{s,l}|}\int_{R_{s,l}}g(x-y)dy,
\Ees
where $\lc$ should be understood as partial order in the positive cone of $L_2(\N)$.
For convenience, we set $M_{R_{s,l}}g(x)={|R_{s,l}|^{-1}}\int_{R_{s,l}}g(x-y)dy$. Notice that $M_{R_{s,l}}g=M_{\mathcal{R}_{2^{k/2}}}g$ the Kakeya average operator defined in \eqref{d:21Kak}. Now using the estimate of the Kakeya maximal function in Theorem \ref{t:21Kak}, we get
\Bes
\begin{split}
\|\sup_{l\in\mathfrak{I}}|\F[m_{k,l}]|*g\|_{L_2(\N)}&
\lc(1+|\lam|)^3\sum_{s=0}^\infty2^{-s}\|\sup_{s\in\mathbb{N}, l\in\mathfrak{I}}M_{R_{s,l}}g\|_{L_2(\N)}\\
&\lc(1+|\lam|)^3\sum_{s=0}^\infty 2^{-s}\|\sup_{}M_{\mathcal{R}_{2^{k/2}}}g\|_{L_2(\N)}\\
&\lc(1+|\lam|)^3(1+k)\|g\|_{L_2(\N)},
\end{split}
\Ees
which is the just required estimate of \eqref{e:21l2max}. The proof of \eqref{e:21l2maxself} is similar, we omit the details here.

We turn to the proofs of  \eqref{e:21l4} and \eqref{e:21l4self}.
Recall the strips
\Bes
S_{k,\si,\ups}=\R\times[40\si2^{-\fr k2}+\ups2^{-\fr k2},40(\si+1)2^{-\fr k2}+\ups 2^{-\fr k2}]
\Ees
which is defined for $\si\in\Z$ and $\ups\in\{0,1,\cdots,39\}$. These strips have width $40\cdot 2^{-\fr k2}$ and each $\tilde{J}_{k,l}$ belongs to one of them, which we call $S_{k,\si_l,\ups_l}=B_{k,l}$.

The family $\{B_{k,l}\}_{l\in\mathfrak{I}}$ may have duplicate sets, so we split it into $40$ subfamilies by placing $B_{k,l}$ into different subfamilies if the indices $\ups_l$ and $\ups_{l'}$ are different. Thus we write the set $\mathfrak{I}$ as $\mathfrak{I}=\mathfrak{I}^1\cup \mathfrak{I}^2\cup\cdots\cup\mathfrak{I}^{40}$ where elements in each $\mathfrak{I}^i$ are different.

We next observe that the multiplier operator
\Bes
\widehat{f_{k,l}}=\chi_{B_{k,l}}
\widehat{f}\Ees
satisfies $f_{k,l}=(id_{x_1}\otimes P_l)f$ where $id_{x_1}$ is the identity operator in $x_1$ variable and $P_l$ is an operator with the multiplier $\chi_{\{\xi_2: (\xi_1,\xi_2)\in B_{k,l}\}}$. Now using Lemma \ref{t:21Little}, we get for $2<p<\infty$,
\Bes
\begin{split}
\big\|\{f_{k,l}\}_{l\in\mathfrak{I}^i} \big\|_{L_p(\N;\ell_2^{rc})}&=
\big\|\{(id_{x_1}\otimes P_l)(f)\}_{l\in\mathfrak{I}^i} \big\|_{L_p(\N;\ell_2^{rc})}\\
&=\Big(\int_{\R}\big\|\{P_l(f(x_1,\cdot))\}_{l\in\mathfrak{I}^i}\Big\|^p_{L_p(L_\infty(\R)\tens\M;\ell_2^{rc})} dx_1\Big)^{\fr 1p}\\
&\lc\Big(\int_{\R}\|f(x_1,\cdot)\|^p_{L_p(L_\infty(\R)\tens\M)} dx_1\Big)^{\fr 1p}=\|f\|_{L_p(\N)}.
\end{split}\Ees
Specially the case $p=4$ is our required estimate for \eqref{e:21l4} and \eqref{e:21l4self}. Now we prove \eqref{e:21l4} and \eqref{e:21l4self} for one $\mathfrak{I}^i$. Since $\mathfrak{I}=\mathfrak{I}^1\cup \mathfrak{I}^2\cdots \cup \mathfrak{I}^{40}$, the full forms of \eqref{e:21l4} and \eqref{e:21l4self} just follow from Minkowski's inequality.
\end{proof}

In the remain part of this section, we turn to Lemma \ref{l:21Geo}. This lemma is stated in \cite[Exercise 3.4]{Dem20} without a proof. For the sake of self-containment, we give a proof here.

\begin{proof}[{Proof of Lemma \ref{l:21Geo}}]
We only need to show that for each integer $k$, the following estimate
$$\sum_{l\in\mathfrak{I}}\sum_{|l-l'|>10^3}\chi_{\Ga_{k,l}-{\Ga}_{k,l'}}(\xi)\leq C$$
holds.
If $k$ is finite, there exist only finite many pairs of sets $\Ga_{k,l}-\Ga_{k,l'}$ depending on $k$ and this lemma is trivial.
Therefore we can assume that $k$ is a large integer. Set $\del=2^{-k}$. Then $\del$ is sufficient small. For simplicity, denote the set $\Ga_{k,l}$ by $\Ga_l$.
Elements of $\Ga_l-\Ga_{l'}$ have the form
\Bes
E:=re^{2\pi i(l+\alp)\del^{\fr12}}-r'e^{2\pi i(l'+\alp')\del^{\fr12}},
\Ees
where $\alp,\alp'\in[-1,1]$ and $r,r'\in[1-\fr58\del,1-\fr18\del]$. We first give some geometric observations of $\Ga_{l}-\Ga_{l'}$. Set
\Be\label{e:21wllmain}
w(l,l'):=e^{2\pi il\del^{\fr12}}-e^{2\pi il'\del^{\fr12}}=2\sin(\pi(l-l')\del^{\fr12})ie^{\pi i(l+l')\del^{\fr12}},
\Ee
then it is easy to see  that the direction of $w(l,l')$ is along $ie^{\pi i(l+l')\del^{\fr12}}$. Next we rewrite $E$ as follows
\Bes
w(l,l')+\Big[re^{2\pi i(l+\alp)\del^{\fr12}}-re^{2\pi i(l'+\alp')\del^{\fr12}}-rw(l,l')\Big]
+\Big[(r-1)w(l,l')+(r-r')e^{2\pi i(l'+\alp')\del^{\fr12}}
\Big].
\Ees
By some elementary calculations, we further get the second term above equals to
$$rw(l,l')(\fr{e^{2\pi i\alp\del^{\fr12}}+e^{2\pi i\alp'\del^{\fr12}}-2}{2})+r(e^{2\pi il'\del^{\fr12}}+e^{2\pi il\del^{\fr12}})(\fr{e^{2\pi i\alp\del^{\fr12}}-e^{2\pi i\alp'\del^{\fr12}}}{2}).$$
Combining this equality, we see that $E$ equals
\Bes
\begin{split}
w(l,l')&+ir(e^{2\pi il'\del^{\fr12}}+e^{2\pi il\del^{\fr12}})\fr{\sin2\pi\alp\del^{\fr12}-\sin2\pi\alp'\del^{\fr12}}{2}\\
&+r(e^{2\pi il'\del^{\fr12}}+e^{2\pi il\del^{\fr12}})\fr{\cos2\pi\alp\del^{\fr12}-\cos2\pi \alp'\del^{\fr12}}{2}+E(r,r')
\end{split}
\Ees
where
\Bes
\begin{split}
E(r,r')&=(r-1)w(l,l')+(r-r')e^{2\pi i(l'+\alp')\del^{\fr12}}+rw(l,l')\fr{\cos2\pi\alp\del^{\fr12}+\cos2\pi\alp'\del^{\fr12}-2}{2}\\
&\quad+irw(l,l')\fr{\sin2\pi\alp\del^{\fr12}+\sin2\pi\alp'\del^{\fr12}}{2}.
\end{split}\Ees

On the other hand, utilizing the following equality
$$e^{2\pi il'\del^{\fr12}}+e^{2\pi il\del^{\fr12}}=2\cos(\pi (l-l')\del^{\fr12})e^{\pi i(l+l')\del^{\fr12}},$$
we see that $i(e^{2\pi il'\del^{\fr12}}+e^{2\pi il\del^{\fr12}})$ has the same direction as $w(l,l')$. Thus, $(e^{2\pi il'\del^{\fr12}}+e^{2\pi il\del^{\fr12}})$ is perpendicular to $w(l,l')$.
Before proceeding further, we need the following inequalities which will be frequently used in later manipulations:
\Be\label{e:21sincos}
\begin{split}
\rm{(i):\ |\sin t|\leq |t|;\quad (ii):\ |1-\cos t|\leq {|t|}^2/2;\quad (iii):\ |\sin t|\geq\fr{2|t|}{\pi}, \forall |t|\leq\fr{\pi}{2}}.
\end{split}
\Ee
By the first and second inequalities in \eqref{e:21sincos}, it is straightforward to verify that the error term satisfies
\Bes
|E(r,r')|\leq\fr54\del+\fr{\del}{2}+4\pi^2\del+4\pi^2\del|l-l'|\leq50\del+4\pi^2\del|l-l'|.
\Ees
We therefore conclude that $\Ga_{l}-\Ga_{l'}$ contains in a rectangle $R(l,l')$ centered at $w(l,l')$ with half-length
$$4\pi\del^{\fr12}+50\del+4\pi^2\del|l-l'|<100\del^{\fr12}$$
in the direction along $w(l,l')$ since $\del$ is sufficient small and half-width
$$4\pi^2\del+50\del+4\pi^2\del|l-l'|<90|l-l'|\del<100\del^{\fr12}$$
in the direction along $i w(l,l')$ since $|l-l'|>10^3$. Moreover $R(l,l')\subset B(w(l,l'),150\del^{\fr12})$ a disk centered at $w(l,l')$ with radius $150\del^{\fr12}$.

Our next goal is to show that for a fixed $(l,l')\in(\mathfrak{I},\mathfrak{I})$ with $|l-l'|>10^3$, there exist only finite pairs $(m,m')\in(\mathfrak{I},\mathfrak{I})$ with $|m-m'|>10^3$ such that $(\Ga_m-\Ga_{m'})\cap(\Ga_l-\Ga_{l'})\neq\emptyset$. Once we prove this fact, the geometric estimate in Lemma \ref{l:21Geo} immediately follows from it.
If $|w(l,l')-w(m,m')|>300\del^{\fr12}$, then $(\Ga_m-\Ga_{m'})\cap(\Ga_l-\Ga_{l'})=\emptyset$. Therefore if these two sets intersect, we get
\Be\label{e:21wllup}
|w(l,l')-w(m,m')|\leq 300\del^{\fr12}.
\Ee

In the following we should give a lower bound for $|w(l,l')-w(m,m')|$. Before that we need a fundamental inequality which could be found in \cite[Exercise 5.2.2]{Gra14M}:
\Be\label{e:21cos}
|r_1e^{i\tet_1}-r_2e^{i\tet_2}|\geq\min(r_1,r_2)\sin|\tet_1-\tet_2|,\quad \forall\ |\tet_1|,|\tet_2|<\fr{\pi}{4}.
\Ee
By some elementary calculations and using the inequality \eqref{e:21cos} with (iii) of \eqref{e:21sincos}, we get a lower bound as follows
\Bes
\begin{split}
|w(l,l')-w(m,m')|&=|2\cos(\pi(l-m')\del^{\fr12})e^{\pi i(l+m')\del^{\fr12}}-2\cos(\pi(l'-m)\del^{\fr12})e^{\pi i(l'+m)\del^{\fr12}}|\\
&\geq\fr{4}{\pi}\cos(\fr\pi4)|\pi[((l+m')-\pi(l'+m))]|\del^{\fr12}
\end{split}\Ees
where we use our hypothesis that all the supports of $m_{k,l}$ are contained in $Q_a$ (i.e. $|2\pi l\del^{\fr12}|, |2\pi l'\del^{\fr12}|, |2\pi m\del^{\fr12}|, |2\pi m'\del^{\fr12}|<\fr{\pi}{4}$). Thus combining the above estimate and \eqref{e:21wllup}, we obtain
\Be\label{e:21llup}
|(l-l')-(m-m')|<200.
\Ee
However \eqref{e:21llup} is not enough to show that the number of $(m,m')$ is finite. In the following, we give more exact information of lower and upper bound of $|w(l,l')-w(m,m')|$.

According our condition, $|l-l'|>10^3$ and $|m-m'|>10^3$. Without loss of generality, we suppose that $l>l',m>m'$. Then by \eqref{e:21llup}, we get
\Bes
l-l'\approx m-m'.
\Ees
This implies $|w(l,l')|\approx|w(m,m')|\approx (l-l')\del^{\fr12}$ in view of the identity \eqref{e:21wllmain} and (i), (iii) of \eqref{e:21sincos}.

Next we see that the rectangle $R({l,l'})$ is centered at point $w(l,l')$ with a distance to the origin $|w(l,l')|=2\sin(\pi(l-l')\del^{\fr12})$.
Recall that $R({l,l'})$ has half-length $100\del^{\fr12}$ in the direction $w(l,l')$ and in the direction $iw(l,l')$ has half-width $90(l-l')\del$ (see Figure \ref{f:21inter} below). This implies that $R({l,l'})$ is far from the origin $(0,0)$. In fact, by (iii) of \eqref{e:21sincos},
\Bes
dist(R(l,l'),(0,0))\geq2\sin(\pi(l-l')\del^{\fr12})-100\del^{\fr12}\geq3{(l-l')\del^{\fr12}}.
\Ees

Similar properties also hold for $R({m,m'})$. Since $\Ga_{l}-\Ga_{l'}$ intersects $\Ga_{m}-\Ga_{m'}$, then $R({l,l'})$ intersects $R(m,m')$.
Note that the angle $\tet$ between $w(l,l')$ and $w(m,m')$ is $\pi|l+l'-m-m'|\del^{\fr12}$ (see Figure \ref{f:21inter} below).

\begin{figure}[!h]
\centering
\includegraphics[height=0.4\textwidth]{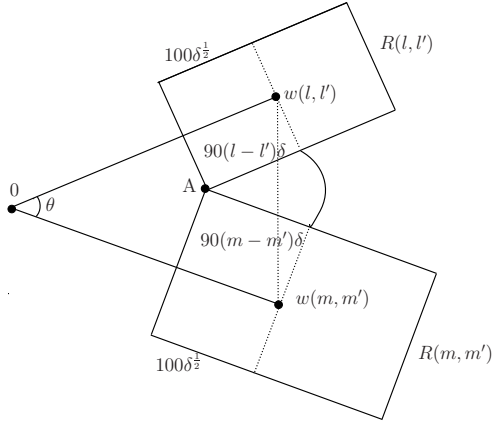}
\caption{\small $R({l,l'})$ has half-length $100\del^{\fr12}$ in the direction along $w(l,l')$ and in the direction along $iw(l,l')$ has half-width $90(l-l')\del$. $R(l,l')$ intersects $R(m,m')$ at the point $A$.}\label{f:21inter}
\end{figure}

Then the distance of $w(l,l')$ and $w(m,m')$ has an upper bound
$$|w(l,l')-w(m,m')|\leq90(l-l')\del+90(m-m')\del+100\del^{\fr12}\tet.$$
On the other hand, the distance of $w(l,l')$ and $w(m,m')$ has the following lower bound
\Bes
\begin{split}
|w(l,l')-w(m,m')|&\geq2\Big(\sin(\pi(l-l')\del^{\fr12})+\sin(\pi(m-m')\del^{\fr12})\Big)\sin\tfrac{\tet}{2}\\
&\geq 4((l-l')+(m-m'))|l+l'-m-m'|\del,
\end{split}\Ees
where we use (iii) of \eqref{e:21sincos}. The upper and lower estimates, together with our hypothesis $l-l'>10^3$, $m-m'>10^3$, finally yield that
\Be\label{e:21lllow}
|l+l'-m-m'|\leq200.
\Ee
Combining \eqref{e:21lllow} and \eqref{e:21llup}, we get that $|l-m|<200$, $|l'-m'|<200$. Hence we prove that for a fixed $(l,l')$, there exist only finite pairs $(m,m')$ such that $(\Ga_{l}-\Ga_{l'})\cap(\Ga_m-\Ga_{m'})\neq\emptyset$,
which completes our proof.
\end{proof}

\vskip0.24cm

\section{Bochner-Riesz means on quantum tori}\label{s:21appli}
In this section, our goal is to establish the full boundedness of Bochner-Riesz means on two-dimensional quantum tori, i.e. Theorem \ref{t:21BRqtorusi}.
We first establish the corresponding results on usual torus based on Theorem \ref{t:21BR} in Section \ref{s:21BR}.
\subsection{Bochner-Riesz means on usual tori}
Recall that the $d$-torus $\T^d$ is defined as $\R^d/\Z^d$. We often set $\T^d$ as the cube $[0,1]^d$ with opposite sides identified. The function on $\T^d$ can be regarded as an $1$-periodic function on $\R^d$ in every coordinate. Haar measure on $\T^d$ is the restriction of Lebesgue measure to $[0,1]^d$ which is still denoted by $dx$.

Our previous main theorem of Bochner-Riesz means in Section \ref{s:21BR} is in the frame of the tensor von Neumann algebra $L_\infty(\R^d)\tens\M$. To extend our main result to the fully noncommutative quantum torus, we should first transfer them  to the setting of the tensor von Neumann algebra $L_\infty(\T^d)\tens\M$.

Given a function $f:\T^d\rta\M$, define the Bochner-Riesz means $\B^\lam_R(f)(x)$  by
\eqref{e:21BRtorus}.
By the transference method (see Theorem \ref{t:21transft} later), we can formulate the noncommutative Bochner-Riesz conjecture on tori with ranges of indexes $r$ and $p$ the same as that in the commutative case in the introduction.
Our main result in this subsection is stated as follows.
\begin{theorem}\label{t:21BRtorus}
Suppose $0<\lam\leq\fr12$ and $\fr{4}{3+2\lam}<p<\fr{4}{1-2\lam}$. Let $\B_R^\lam$ be defined in \eqref{e:21BRtorus} for $d=2$. Then we have
\Bes
\sup_{R>0}\|\B_R^\lam(f)\|_{L_p(L_\infty(\T^2)\tens\M)}\lc\|f\|_{L_p(L_\infty(\T^2)\tens\M)}.
\Ees
Consequently for $f\in L_p(L_\infty(\T^2)\tens\M)$, $\B_R^\lam(f)$ converges to $f$ in $L_p(L_\infty(\T^2)\tens\M)$ as $R\rta\infty$.
\end{theorem}

To prove Theorem \ref{t:21BRtorus}, we should introduce some definitions and lemmas.
\begin{definition}
For $x\in\R^d$, a bounded function $m$ on $\R^d$ taking values in  $\C$ is called regulated at the point $x$ if
\Bes
\lim_{\eps\rta0}\fr{1}{\eps^d}\int_{|t|\leq\eps}(m(x-t)-m(x))dt=0.
\Ees
The function $m$ is called regulated if it is regulated at every $x\in\R^d$.
\end{definition}

The above definition of regulated function was first appeared in \cite{deL65}. Next we introduce the Fourier multiplier on $L_\infty(\T^d)\tens\M$.

\begin{definition}
Given a bounded function $m:\Z^d\rta \C$, we say $\{m(z)\}_{z\in\Z^d}$ is a Fourier multiplier on $L_p(L_\infty(\T^d)\tens\M)$ if the operator $S$ defined by
\Bes
{S(f)}(x)=\sum_{z\in\Z^d}m(z)\widehat{f}(z)e^{2\pi ix\cdot z}
\Ees
where $f:\T^d\rta\M$ and $\widehat{f}(z)=\int_{\T^d}e^{-2\pi izx}f(x)dx$, extends to a bounded operator on $L_p(L_\infty(\T^d)\tens\M)$.
The space of all these multipliers is denoted by $M_p(L_\infty(\T^d)\tens\M)$.
For simplicity, we denote the norm of $L_p(L_\infty(\T^d)\tens\M)$ multiplier by $\|\cdot\|_{M_p(L_\infty(\T^d)\tens\M)}$.
\end{definition}

\begin{theorem}\label{t:21transft}
Let $m$ be a regulated function which lies in $M_p(L_\infty(\R^d)\tens\M)$ for some $1\leq p<\infty$. Then the sequence $\{m(z)\}_{z\in\Z^d}$ belongs to $M_p(L_\infty(\T^d)\tens\M)$ and more precisely
\Bes
\|\{m(z)\}_{z\in\Z^d}\|_{M_p(L_\infty(\T^d)\tens\M)}\leq \|m\|_{M_p(L_\infty(\R^d)\tens\M)}.
\Ees
Moreover for all $R>0$, the sequence $\{m(z/R)\}_{z\in\Z^d}$ lies in $M_p(L_\infty(\T^d)\tens\M)$ and
\Bes
\sup_{R>0}\|\{m(z/R)\}_{z\in\Z^d}\|_{M_p(L_\infty(\T^d)\tens\M)}\leq \|m\|_{M_p(L_\infty(\R^d)\tens\M)}.
\Ees
\end{theorem}
\begin{proof}
The proof of this theorem is quite similar to that in the commutative case (see e.g. the proof of Theorem 4.3.7 in \cite{Gra14C}), so we omit the details of the proof here.
\end{proof}

Now we can apply Theorem \ref{t:21BR} and Theorem \ref{t:21transft} to show Theorem \ref{t:21BRtorus}.
Indeed notice that the function $m(z)=(1-|z|^2)_+^\lam$ is continuous on $\R^d$. Hence it is regulated. Now utilizing Theorem \ref{t:21BR} and Theorem \ref{t:21transft}, it is easy to obtain Theorem \ref{t:21BRtorus}.

\subsection{Bochner-Riesz means on quantum tori}\label{s:21appliquan}
In this subsection, we finally consider the Bochner-Riesz means on two-dimensional quantum tori and prove our main results Theorem \ref{t:21BRqtorusi}. We begin by introducing some notation.

Suppose that $d\geq 2$, $\tet=(\tet_{k,j})_{1\leq k,j\leq d}$ is a real skew symmetric $d\times d$ matrix.
The $d$-dimensional noncommutative torus $\A_\tet$ is a universal $C^*$-algebra generated by $d$ unitary operators $U_1,\cdots,U_d$ satisfying the following commutation relation:
\Bes
U_kU_j=e^{2\pi i\tet_{k,j}}U_jU_k,\quad 1\leq k,j\leq d.
\Ees
By the unitary property of $U_k$s, if we multiply $U^*_k$ in both left and right sides of the above equality, we get
\Bes
U_jU^*_k=e^{2\pi i\tet_{k,j}}U_k^*U_j,\quad 1\leq k,j\leq d.
\Ees
For a multi-index $k=(k_1,\cdots,k_d)\in\Z^d$ and $U=(U_1,\cdots,U_d)$, we define $U^k=U_1^{k_1}\cdots U_d^{k_d}$. We call
\Be\label{e:21quanpoly}
f=\sum_{k\in\Z^d}\alp_k U^k,\quad\text{with}\ \  \alp_k\in \C,
\Ee
is a polynomial in $U$ if the sum in \eqref{e:21quanpoly} is finite, i.e. $\alp_k\neq 0$ holds only for finite multi-indexes.
Let $\mathcal{P}_\tet$ be the involution algebra of all such polynomials. Then $\mathcal{P}_\tet$ is dense in $\A_\tet$.
A well-known fact is that $\A_\tet$ admits a faithful tracial state $\tau$ such that $\tau(U_1^{k_1}\cdots U_d^{k_d})=1$ if and only if $k=(k_1,\cdots,k_d)=\textbf{0}$ where $\textbf{0}=(0,\cdots,0)$ (see e.g. \cite{Rie90}).
Hence for any polynomial $f$ with form  \eqref{e:21quanpoly}, we can define
\Bes
\tau(f)=\alp_{\textbf{0}}.
\Ees
Define $\T_\tet^d$ the weak $*$-closure of $\A_\tet$ in the GNS representation of $\tau$.
We call $\T_\tet^d$ the $d$-dimensional quantum torus associated to $\tet$. The state $\tau$ also extends to a normal faithful state on $\T_\tet^d$, which will be denoted again by $\tau$.
Notice that when $\tet=0$, $\A_\tet=C(\T^d)$ and $\T_\tet^d=L_\infty(\T^d)$. Thus quantum torus $\T_\tet^d$ is a deformation of classical torus $\T^d$.

Let $L_p(\T_\tet^d)$ be the noncommutative space associated with $(\T_\tet^d,\tau)$.
Using $\tau$ is a state and H\"older's inequality, we see that $L_q(\T_\tet^d)\subset L_p(\T^d_\tet)$ for $0<p<q<\infty$.
For any $f\in L_1(\T_\tet^d)$, there exists a formal Fourier series
\Bes
f\sim\sum_{m\in\Z^d}\wh{f}(m)U^m
\Ees
where $\wh{f}(m)=\tau((U^m)^*f)$ is called the Fourier coefficient of $f$.
Analogous to the classical analysis, a fundamental problem here is that when the Fourier series
$\sum_{m\in\Z^d}\wh{f}(m)U^m$ converges to $f$.
In the following, we consider the most important Bochner-Riesz means on quantum tori which is defined by
\Be\label{e:21BRqtorus}
\B^\lam_R(f)=\sum_{m\in\Z^d}\big(1-|\tfrac{m}{R}|^2\big)^\lam_+\hat{f}(m)U^m.
\Ee
The Bochner-Riesz means on quantum tori was firstly studied by Z. Chen, Q. Xu and Z. Yin \cite{CXY13}. A fundamental problem raised in \cite[Page 762]{CXY13} is that in which senses $\B^\lam_R(f)$ converge back to $f$. We consider the $L_p$ convergence here. In fact similar to that of the commutative case, one can formulate the following problem of quantum Bochner-Riesz means.
\begin{conj}
Suppose $\lam>0$ and $\fr{2d}{d+1+2\lam}<p<\fr{2d}{d-1-2\lam}$. Consider the Bochner-Riesz means defined in \eqref{e:21BRqtorus}, then we have
\Bes
\sup_{R>0}\|\B_R^\lam(f)\|_{L_p(\T_\tet^d)}\lc\|f\|_{L_p(\T_\tet^d)}.
\Ees
\end{conj}

The main result here is to show that the preceding conjecture holds for two dimensions, i.e. Theorem \ref{t:21BRqtorusi}. For the reader's convenience, we  restate this theorem as follows.
\begin{theorem}\label{t:21BRqtorus}
Suppose $0<\lam\leq\fr12$ and $\fr{4}{3+2\lam}<p<\fr{4}{1-2\lam}$. Let the Bochner-Riesz means $\B_R^\lam$ be defined in \eqref{e:21BRqtorus} for $d=2$. Then we have
\Bes
\sup_{R>0}\|\B_R^\lam(f)\|_{L_p(\T_\tet^2)}\lc\|f\|_{L_p(\T^2_\tet)}.
\Ees
Consequently for $f\in L_p(\T_\tet^2)$, $\B_R^\lam(f)$ converges to $f$ in $L_p(\T_\tet^2)$ as $R\rta\infty$.
\end{theorem}

The proof is based on a transference technique which is a standard method.
Consider the tensor von Neumann algebra $\N_\tet=L_\infty(\T^d)\tens\T_\tet^d$ equipped with the tensor trace $v=\int_{\T^d} dx\otimes\tau$.
Let $L_p(\N_\tet)$ be the noncommutative $L_p$ space associated with $(\N_\tet,v)$. Observe that
\Bes
L_p(\N_\tet)\cong L_p(\T^d;L_p(\T^d_\tet))
\Ees
where the space on the right hand side is the Bochner $L_p$ space on $\T^d$ with values in $L_p(\T_\tet^d)$.

For every $x\in\Z^d$, we define $\pi_x$ as
\Bes
\pi_{x}(U^k)=e^{2\pi ix\cdot k}U^k=e^{2\pi ix_1k_1}\cdots e^{2\pi ix_dk_d}U^{k_1}\cdots U^{k_d}
\Ees
which is an isomorphism of $\T_\tet^d$. It is easy to see that $\tau(\pi_x(f))=\tau(f)$ for any $x\in\T^d$. Hence $\pi_x$ is trace preserving. Therefore it extends to an isometry on $L_p(\T^d_\tet)$ for $1\leq p<\infty$, i.e.
\Be\label{e:21isome}
\|\pi_x(f)\|_{L_p(\T_\tet^d)}=\|f\|_{L_p(\T^d_\tet)}.
\Ee

The following transference method has been showed by Z. Chen,  Q. Xu and Z. Yin in \cite[Proposition 2.1 \& Corollary 2.2]{CXY13}.

\begin{lemma}\label{l:21transq}
For any $f\in L_p(\T^d_\tet)$, the function $\tilde{f}:x\rta\pi_x(f)$ is continuous from $\T^d$ to $L_p(\T^d_\tet)$ (with respect to the weak $*$-topology for $p=\infty$).
Moreover $\tilde{f}\in L_p(\N_\tet)$ and $\|\tilde{f}\|_{L_p(\N_\tet)}=\|f\|_{L_p(\T^d_\tet)}$ for $1\leq p\leq\infty$. Thus, $f\rta\tilde{f}$ is an isometric embedding from $L_p(\T_\tet^d)$ to $L_p(\N_\tet)$.
\end{lemma}

\begin{proof}[{Proof of Theorem \ref{t:21BRqtorus}}]
The proof just follows from Theorem \ref{t:21BRtorus} and Lemma \ref{l:21transq}.
In fact, by the density argument, it suffices to consider $f$ as a polynomial $\sum_{k\in\Z^2}\hat{f}(k) U^k$.  Define $\tilde{f}$ in Lemma \ref{l:21transq}. Then $\tilde{f}\in L_p(\N_\tet)$ and $\|\tilde{f}\|_{L_p(\N_\tet)}=\|f\|_{L_p(\T^2_\tet)}$ by Lemma \ref{l:21transq}.
Using Theorem \ref{t:21BRtorus}, we get
\Be\label{e:21suprqu}
\sup_{R>0}\|\B_R^\lam(\tilde{f})\|_{L_p(\T^2;L_p(\T_\tet^2))}\lc\|\tilde{f}\|_{L_p(\T^2;L_p(\T_\tet^2))}.
\Ee
On the other hand, it is easy to see $\widehat{\tilde{f}}(k)=\wh{f}(k)U^k$. According the definition of Bochner-Riesz means,
\Bes
\begin{split}
\B^\lam_R(\tilde{f})(x)=\sum_{k\in\Z^2}\big(1-\big|\tfrac{k}{R}\big|^2\big)_+^\lam\widehat{\tilde{f}}(k)e^{2\pi ik\cdot x}=\sum_{k\in\Z^2}\big(1-\big|\tfrac{k}{R}\big|^2\big)_+^\lam\widehat{f}(k)e^{2\pi ik\cdot x}U^k
=\pi_x[\B^\lam_R(f)].
\end{split}
\Ees
This, together with \eqref{e:21isome} and \eqref{e:21suprqu}, implies the desired estimate in Theorem \ref{t:21BRqtorus}. The convergence follows from the standard limiting argument.
\end{proof}
\vskip0.24cm

\appendix
\section{Interpolation of analytic families of operators on noncommutative $L_p$ spaces}

In this appendix, we state  precisely an analytic interpolation theorem which may be known to experts.
Let $\S$ be the linear span of all $x\in\M_{+}$ whose support projections have finite trace.
Suppose that $T_z$ is a linear operator mapping $\mathcal{S}$ to itself for every  $z$ in the closed strip $\bar{S}=\{z\in\C:0\leq \text{Re} z\leq1\}$.
We say the family $\{T_z\}_z$ is analytic if the function
$$z\rta\tau(gT_z(f))$$
is analytic in the open strip $S=\{z\in\C:0<\Re z<1\}$ and continuous on $\bar{S}$ for any functions $f$ and $g$ in $\mathcal{S}$.
Moreover we say the analytic family $\{T_z\}_z$ is of admissible growth if there exists a constant $0<a<\pi$ such that
\Bes
e^{-a|\Im z|}\log|\tau(gT_z(f))|<\infty
\Ees
for all $z\in\bar{S}$. Now we can state the following analytic interpolation theorem.
\begin{theorem}\label{t:21inter}
Suppose that $T_z$ is an analytic family of linear operators of admissible growth. Let $p_0,p_1,q_0,q_1\in(0,\infty)$ and assume that $M_0,M_1$ are positive functions on $\R$ such that
\Be
\sup_{y\in\R}e^{-b|y|}\log M_j(y)<\infty
\Ee
for $j=0,1$ and some $b\in(0,\pi)$. Let $p,\tet$ satisfy $\fr{1}{p}=\fr{1-\tet}{p_0}+\fr{\tet}{p_1}$ and $\fr{1}{q}=\fr{1-\tet}{q_0}+\fr{\tet}{q_1}$.
Suppose that
\Be
\begin{split}
\|T_{iy}(f)\|_{L_{q_0}(\M)}&\leq M_0(y)\|f\|_{L_{p_0}(\M)},\quad
\|T_{1+iy}(f)\|_{L_{q_1}(\M)}\leq M_1(y)\|f\|_{L_{p_1}(\M)}
\end{split}
\Ee
hold for all $f\in\mathcal{S}$. Then for any $\tet\in(0,1)$, we have
\Bes
\|T_{\tet}(f)\|_{L_q(\M)}\leq M(\tet)\|f\|_{L_p(\M)},
\Ees
where for $0<t<1$,
\Bes
M(t)=\exp\Big\{\fr{\sin(\pi t)}{2}\int_{-\infty}^{\infty}\Big[\fr{\log M_0(y)}{\cosh(\pi y)-\cos(\pi t)}+\fr{\log M_1(y)}{\cosh(\pi y)+\cos(\pi t)}\Big]dy\Big\}.
\Ees
\end{theorem}

To prove this theorem, we need an extension of the three lines theorem  which could be found in \cite[Page 206, Lemma 4.2]{SW71}.
\begin{lemma}\label{l:21ethreeline}
Suppose that $F$ is analytic on the open strip $S$ and continuous on its closure such that
\Bes
\sup_{z\in\bar{S}}e^{-a|\Im z|}\log |F(z)|<\infty
\Ees
for some $a\in(0,\pi)$. Then for any $0<x<1$, we have
\Bes
|F(x)|\leq\exp\Big\{\fr{\sin(\pi x)}{2}\int_{-\infty}^{\infty}\Big[\fr{\log |F(iy)|}{\cosh(\pi y)-\cos(\pi x)}+\fr{\log |F(1+iy)|}{\cosh(\pi y)+\cos(\pi x)}\Big]dy\Big\}.
\Ees
\end{lemma}

\begin{proof}[Proof of Theorem \ref{t:21inter}]
The proof is quite similar to that in the commutative case. Let $f,g\in\mathcal{S}$ with polar decompositions $f=u|f|$ and $g=v|g|$.
Without loss of generality, we may suppose that $\|f\|_{L_p(\M)}=1=\|g\|_{L_{q'}(\M)}$. By the duality, to prove our theorem, it suffices to show
\Bes
|\tau(gT_{\tet}(f))|\leq M({\tet}).
\Ees
For $z\in\bar{S}$, define $f(z)=u|f|^{\fr{p(1-z)}{p_0}+\fr{pz}{p_1}}$ and $g(z)=v|g|^{\fr{q'(1-z)}{q'_0}+\fr{q'z}{q'_1}}$ where the continuous functional calculus is defined by complex powers of positive operators. By the density argument, we could suppose that $|f|$ and $|g|$ are linear combinations of mutually orthogonal projections of finite trace, i.e.
\Bes
|f|=\sum_{j=1}^n\alp_je_j,\quad |g|=\sum_{k=1}^m\beta_k\tilde{e}_k
\Ees
where $\alp_j$s, $\beta_k$s are real and $e_j$s, $\tilde{e}_k$s are mutually orthogonal basis. Then
\Bes
f(z)=\sum_{j=1}^n\alp_j^{\fr{p(1-z)}{p_0}+\fr{pz}{p_1}}ue_j.
\Ees
Therefore the function $z\rta f(z)$ is an analytic function on $\C$ taking values in $\M$. Similar properties hold for the function $z\rta g(z)$.
Define
\Bes
F(z)=\tau(g(z)T_z(f(z))).
\Ees
Then we have
\Bes
F(z)=\sum_{j=1}^n\sum_{k=1}^m\alp_j^{\fr{p(1-z)}{p_0}+\fr{pz}{p_1}}
\beta_k^{\fr{q'(1-z)}{q'_0}+\fr{q'z}{q'_1}}\tau(v\tilde{e}_kT_z(ue_j)).
\Ees

By our assumption, $\tau(v\tilde{e}_kT_z(ue_j))$ is analytic. Hence $F(z)$ is an analytic function satisfying the hypothesis of Lemma \ref{l:21ethreeline}.  Recall a property of polar decomposition: $|f|=|f|u^*u=u^*u|f|$, then by the continuous functional calculus of $|f|$, we obtain
\Be\label{e:21contuf}
\om(|f|)=\om(|f|)u^*u=u^*uw(|f|),
\Ee
where $\om$ is a continuous function on $\R_+$.
Since
\Bes
f(iy)=u|f|^{iyp(\fr{1}{p_1}-\fr{1}{p_0})+\fr{p}{p_0}},
\Ees
then by \eqref{e:21contuf}, we get
\Bes
\begin{split}
|f(iy)|^2=f^*(iy)f(iy)=|f|^{-iyp(\fr{1}{p_1}-\fr{1}{p_0})+\fr{p}{p_0}}
u^*u|f|^{iyp(\fr{1}{p_1}-\fr{1}{p_0})+\fr{p}{p_0}}
=|f|^{\fr{2p}{p_0}}.
\end{split}
\Ees
Therefore we get $\|f(iy)\|_{L_{p_0}(\M)}=1$. Similarly $\|f(1+iy)\|_{L_{p_1}(\M)}=1=\|g(iy)\|_{L_{q'_0}(\M)}=\|g(1+iy)\|_{L_{q'_1}(\M)}$.
H\"older's inequality and our assumption show that for all $y\in\R$,
\Bes
\begin{split}
|F(iy)|&\leq\|T_{iy}(f(iy))\|_{L_{q_0}(\M)}\|g(iy)\|_{L_{q'_0}(\M)}\\
&\leq M_0(y)\|f(iy)\|_{L_{p_0}(\M)}\|g(iy)\|_{L_{q'_0}(\M)}=M_0(y).
\end{split}
\Ees
Similarly for all $y\in\R$, $|F(1+iy)|\leq M_1(y)$. Now applying Lemma \ref{l:21ethreeline} with the preceding two estimates and notice that $\cosh(\pi y)=\fr12(e^{\pi y}+e^{-\pi y})\geq1\geq\cos(\pi x)$, we get
\Bes
|\tau(gT_\tet(f))|
=|F(\tet)|\leq M(\tet),
\Ees
which implies the required estimate.
\end{proof}

\subsection*{Acknowledgement} The author would like to thank Guixiang Hong for some helpful suggestions when preparing this paper and the referees for their very careful reading and valuable suggestions.

\vskip1cm

\bibliographystyle{amsplain}
\bibliography{rf}

\end{document}